\documentclass{conm-p-l}
\usepackage{amssymb}

\usepackage{amsfonts}
\usepackage{eucal}
\usepackage{upgreek}
\begin{document}

\font\eightrm=cmr8
\font\eightit=cmti8
\font\eighttt=cmtt8
\def\tci
{\hbox{\hskip1.8pt$\rightarrow$\hskip-11.5pt$^{^{C^\infty}}$\hskip-1.3pt}}
\def\nft
{\hbox{$n$\hskip3pt$\equiv$\hskip4pt$5$\hskip4.4pt$($mod\hskip2pt$3)$}}
\def\bbR{\mathrm{I\!R}}
\def\rto{\bbR\hskip-.5pt^2}
\def\rtr{\bbR\hskip-.7pt^3}
\def\rfo{\bbR\hskip-.7pt^4}
\def\rn{\bbR^{\hskip-.6ptn}}
\def\mr{\bbR^{\hskip-.6ptm}}
\def\bbZ{\mathsf{Z\hskip-4ptZ}}
\def\bbRP{\text{\bf R}\text{\rm P}}
\def\bbC{\text{\bf C}}
\def\cx{C\hskip-2pt_x\w}
\def\cy{C\hskip-2pt_y\w}
\def\cz{C\hskip-2pt_z\w}
\def\hyp{\hskip.5pt\vbox
{\hbox{\vrule width3ptheight0.5ptdepth0pt}\vskip2.2pt}\hskip.5pt}
\def\er{r}
\def\es{s}
\def\df{d\hskip-.8ptf}
\def\fv{\mathcal{F}}
\def\fvp{\fv_{\nrmh p}}
\def\wv{\mathcal{W}}
\def\vt{\mathcal{P}}
\def\tv{\mathcal{T}}
\def\vtx{\vt_{\nh x}}
\def\fh{f}
\def\g{\mathtt{g}}
\def\rc{\theta}
\def\jm{\mathcal{I}}
\def\ke{\mathcal{K}}
\def\xc{\mathcal{X}_c}
\def\lz{\mathcal{L}}
\def\dla{\mathcal{D}_{\hskip-2ptL}^*}
\def\dxa{\mathcal{D}_{\hskip-1.3ptx}^*}
\def\Lie{\pounds}
\def\lv{\Lie\hskip-1.2pt_v\w}
\def\lo{\lz_0}
\def\xe{\mathcal{E}}
\def\eo{\xe_0}
\def\lsq{\mathsf{[}}
\def\rsq{\mathsf{]}}
\def\hm{\hskip1.9pt\widehat{\hskip-1.9ptM\hskip-.2pt}\hskip.2pt}
\def\hmt{\hskip1.9pt\widehat{\hskip-1.9ptM\hskip-.5pt}_t}
\def\hmz{\hskip1.9pt\widehat{\hskip-1.9ptM\hskip-.5pt}_0}
\def\hmp{\hskip1.9pt\widehat{\hskip-1.9ptM\hskip-.5pt}_p}
\def\hg{\g}
\def\hk{\hskip1.5pt\widehat{\hskip-1.5ptK\hskip-.5pt}\hskip.5pt}
\def\hq{\hskip1.5pt\widehat{\hskip-1.5ptQ\hskip-.5pt}\hskip.5pt}
\def\txm{{T\hskip-3.5pt_x\w M}}
\def\tyhm{{T\hskip-3.5pt_y\w\hm}}
\def\q{q}
\def\bq{\hat q}
\def\p{p}
\def\w{^{\phantom i}}
\def\x{v}
\def\y{y}
\def\vp{\vt^\perp}
\def\vd{\vt\hh'}
\def\vdx{\vd{}\hskip-4.5pt_x}
\def\bz{b\hh}
\def\fe{F}
\def\fy{\phi}
\def\vl{\Lambda}
\def\hy{\mathcal{V}}
\def\vh{h}
\def\mv{V}
\def\vo{V_{\nnh0}}
\def\ao{A_0}
\def\bo{B_0}
\def\uv{\mathcal{U}}
\def\sv{\mathcal{S}}
\def\svp{\sv_p}
\def\xv{\mathcal{X}}
\def\xvp{\xv_p}
\def\yv{\mathcal{Y}}
\def\yvp{\yv_p}
\def\zv{\mathcal{Z}}
\def\zvp{\zv_p}
\def\cv{\mathcal{C}}
\def\dy{\mathcal{D}}
\def\nv{\mathcal{N}}
\def\iv{\mathcal{I}}
\def\gkp{\Sigma}
\def\ret{\sigma}
\def\taw{\uptau}
\def\hs{\hskip.7pt}
\def\hh{\hskip.4pt}
\def\hn{\hskip-.4pt}
\def\nh{\hskip-.7pt}
\def\nnh{\hskip-1pt}
\def\hrz{^{\hskip.5pt\text{\rm hrz}}}
\def\vrt{^{\hskip.2pt\text{\rm vrt}}}
\def\vt{\varTheta}
\def\op{\varTheta}
\def\vg{\varGamma}
\def\my{\mu}
\def\ny{\nu}
\def\gy{\lambda}
\def\lp{\lambda}
\def\ax{\alpha}
\def\lf{\widetilde{\lp}}
\def\bx{\beta}
\def\ay{a}
\def\by{b}
\def\gp{\mathrm{G}}
\def\hp{\mathrm{H}}
\def\kp{\mathrm{K}}
\def\gm{\gamma}
\def\Gm{\Gamma}
\def\Lm{\Lambda}
\def\Dt{\Delta}
\def\dg{\Delta}
\def\sj{\sigma}
\def\lg{\langle}
\def\rg{\rangle}
\def\lr{\lg\hh\cdot\hs,\hn\cdot\hh\rg}
\def\vs{vector space}
\def\rvs{real vector space}
\def\vf{vector field}
\def\tf{tensor field}
\def\tvn{the vertical distribution}
\def\dn{distribution}
\def\pt{point}
\def\tc{tor\-sion\-free connection}
\def\ea{equi\-af\-fine}
\def\rt{Ric\-ci tensor}
\def\pde{partial differential equation}
\def\pf{projectively flat}
\def\pfs{projectively flat surface}
\def\pfc{projectively flat connection}
\def\pftc{projectively flat tor\-sion\-free connection}
\def\su{surface}
\def\sco{simply connected}
\def\psr{pseu\-\hbox{do\hs-}Riem\-ann\-i\-an}
\def\inv{-in\-var\-i\-ant}
\def\trinv{trans\-la\-tion\inv}
\def\feo{dif\-feo\-mor\-phism}
\def\feic{dif\-feo\-mor\-phic}
\def\feicly{dif\-feo\-mor\-phi\-cal\-ly}
\def\Feicly{Dif-feo\-mor\-phi\-cal\-ly}
\def\diml{-di\-men\-sion\-al}
\def\prl{-par\-al\-lel}
\def\skc{skew-sym\-met\-ric}
\def\sky{skew-sym\-me\-try}
\def\Sky{Skew-sym\-me\-try}
\def\dbly{-dif\-fer\-en\-ti\-a\-bly}
\def\cs{con\-for\-mal\-ly symmetric}
\def\cf{con\-for\-mal\-ly flat}
\def\ls{locally symmetric}
\def\ecs{essentially con\-for\-mal\-ly symmetric}
\def\rr{Ric\-ci-re\-cur\-rent}
\def\kf{Killing field}
\def\om{\omega}
\def\vol{\varOmega}
\def\dv{\delta}
\def\ve{\varepsilon}
\def\zt{\zeta}
\def\kx{\kappa}
\def\mf{manifold}
\def\mfd{-man\-i\-fold}
\def\bmf{base manifold}
\def\bd{bundle}
\def\tbd{tangent bundle}
\def\ctb{cotangent bundle}
\def\bp{bundle projection}
\def\prc{pseu\-\hbox{do\hs-}Riem\-ann\-i\-an metric}
\def\prd{pseu\-\hbox{do\hs-}Riem\-ann\-i\-an manifold}
\def\Prd{pseu\-\hbox{do\hs-}Riem\-ann\-i\-an manifold}
\def\npd{null parallel distribution}
\def\pj{-pro\-ject\-a\-ble}
\def\pd{-pro\-ject\-ed}
\def\lcc{Le\-vi-Ci\-vi\-ta connection}
\def\vb{vector bundle}
\def\vbm{vec\-tor-bun\-dle morphism}
\def\kerd{\text{\rm Ker}\hskip2.7ptd}
\def\ro{\rho}
\def\sy{\sigma}
\def\ts{total space}
\def\pmb{\pi}

\newtheorem{theorem}{Theorem}[section] 
\newtheorem{proposition}[theorem]{Proposition} 
\newtheorem{lemma}[theorem]{Lemma} 
\newtheorem{corollary}[theorem]{Corollary} 
  
\theoremstyle{definition} 
  
\newtheorem{defn}[theorem]{Definition} 
\newtheorem{notation}[theorem]{Notation} 
\newtheorem{example}[theorem]{Example} 
\newtheorem{conj}[theorem]{Conjecture} 
\newtheorem{prob}[theorem]{Problem} 
  
\theoremstyle{remark} 
  
\newtheorem{remark}[theorem]{Remark}

\renewcommand{\theequation}{\arabic{section}.\arabic{equation}}

\title[Compact rank-one ECS manifolds]{The topology of compact rank-one ECS
manifolds}
\author[A. Derdzinski]{Andrzej Derdzinski} 
\address{Department of Mathematics, The Ohio State University, 
Columbus, OH 43210, USA} 
\email{andrzej@math.ohio-state.edu} 
\author[I.\ Terek]{Ivo Terek} 
\address{Department of Mathematics, The Ohio State University, 
Columbus, OH 43210, USA} 
\email{terekcouto.1@osu.edu} 
\def\leftmark{A.\ Derdzinski \&\ I.\ Terek}
\def\rightmark{Compact rank-one ECS manifolds}

\begin{abstract}
Pseu\-do\hs-Riem\-ann\-i\-an manifolds with parallel Weyl tensor that are not
con\-for\-mal\-ly flat or locally symmetric, also known as ECS manifolds,
have a natural local invariant, the rank, which equals 1 or 2, and
is the rank of a certain distinguished null parallel distribution 
$\,\mathcal{D}$. All known examples of compact ECS manifolds are of rank one
and have dimensions greater than 4. We prove that a compact rank-one ECS
manifold, if not locally homogeneous, replaced when necessary by a two-fold
isometric covering, must be a bundle over the circle with 
leaves of $\,\mathcal{D}^\perp$ serving as the fibres. The same conclusion
holds in the lo\-cal\-ly-ho\-mo\-ge\-ne\-ous case if one assumes that 
$\,\mathcal{D}^\perp$ has at least one compact leaf. We also show
that in the pseu\-do\hs-Riem\-ann\-i\-an universal covering space of any
compact rank-one ECS manifold the leaves of $\,\mathcal{D}^\perp$ are the
factor manifolds of a global product decomposition.
\end{abstract}

\maketitle

\setcounter{section}{0}
\setcounter{theorem}{0}
\renewcommand{\thetheorem}{\Alph{theorem}}
\section*{Introduction}
\setcounter{equation}{0}
Pseu\-\hbox{do\hskip.7pt-}Riem\-ann\-i\-an manifolds (or metrics) in
dimensions $\,n\ge4\,$ with parallel Weyl tensor $\,W\hs$ are
often called {\it con\-for\-mal\-ly symmetric\/}
\cite{chaki-gupta}. One speaks of {\it ECS man\-i\-folds/\hn met\-rics\/}
\cite{derdzinski-roter-07} when,
in addition, the metric is neither con\-for\-mal\-ly flat nor locally 
symmetric, 
`ECS' being short for {\it essentially 
con\-for\-mal\-ly symmetric}. 

ECS metrics exist in every dimension $\,n\ge4$, as shown by Roter 
\cite[Corol\-lary~3]{roter}, who also proved that they are all indefinite
\cite[Theorem~2]{derdzinski-roter-77}. The local structure of all 
ECS metrics is described in \cite{derdzinski-roter-09}.

Given an ECS manifold $\,(M\nh,\g)$, we define its {\it rank\/}
\cite{derdzinski-terek-ne} to be the rank $\,d\in\{1,2\}$ of its {\it 
Ol\-szak distribution\/} $\,\mathcal{D}$, which is a null parallel
distribution on $\,M$ discovered by Ol\-szak \cite{olszak}.
The sections of $\,\mathcal{D}\,$ are the vector fields $\,v\,$ having the
property that 
$\,\g(v,\,\cdot\,)\wedge[W(v'\nh,v''\nh,\,\cdot\,,\,\cdot\,)]=0\,$ for all
vector fields $\,v'\nh,v''\nh$. Every Lo\-rentz\-i\-an ECS manifold has rank
one, as the Lo\-rentz\-i\-an signature limits the ranks of null 
distributions to at most $\,1$. For more details, see 
\cite[p.\ 119]{derdzinski-roter-09}. 

Compact rank-one ECS manifolds are known to exist in all dimensions 
$\,n\ge5$, where they represent all indefinite metric signatures
\cite{derdzinski-roter-10,derdzinski-terek-ne}. There are also noncompact
locally homogeneous ECS manifolds \cite{derdzinski-78} of every dimension
$\,n\ge4$. More recently, in \cite{derdzinski-terek-cl}, we constructed
examples of compact locally homogeneous ECS manifolds of all odd dimensions 
$\,n\ge5$.

Our main result can be phrased as follows.
\begin{theorem}\label{maith}
Every non-lo\-cal\-ly-ho\-mo\-ge\-ne\-ous compact rank-one ECS manifold with
trans\-ver\-sal\-ly-o\-ri\-ent\-able distribution\/ $\,\mathcal{D}^\perp\nnh$ 
is dif\-feo\-mor\-phic to a bundle over the circle in such a way that the
fibres coincide with the leaves of $\,\mathcal{D}^\perp\nnh$. This conclusion
remains valid in the lo\-cal\-ly-ho\-mo\-ge\-ne\-ous case, as long as 
$\,\mathcal{D}^\perp\hskip-2pt$ is assumed to have at least one compact leaf.
\end{theorem}
The assertion of Theorem~\ref{maith} obviously implies
that the leaves of $\,\mathcal{D}^\perp\nnh$ are all compact and mutually
dif\-feo\-mor\-phic. 
Note that transversal orientability of $\,\mathcal{D}^\perp\nnh$ can always be
achieved by replacing the manifold in question, if necessary, with a two-fold
isometric covering.

Theorem~\ref{maith} generalizes Theorem B of \cite{derdzinski-roter-08} from
the Lo\-rentz\-i\-an case to any indefinite metric signature. 
The assumption in \cite[Theorem B]{derdzinski-roter-08} does not include rank
one or exclude local homogeneity, since a Lo\-rentz\-i\-an ECS manifold
necessarily has rank one (see above) and cannot be locally homogeneous
(Remark~\ref{lorlh}). The Appendix explains how our proof of 
Theorem~\ref{maith} differs from that used for
\cite[Theorem B]{derdzinski-roter-08}.

The examples of \cite{derdzinski-terek-cl}, mentioned earlier, show that the
final clause of Theorem~\ref{maith} is non-vacuous, at least in odd dimensions.

Triviality of the pull\-back to $\,\bbR\,$ of a bundle over $\,S^1$ makes
the next result an obvious consequence of Theorem~\ref{maith} {\it except
when\/ $\,(\hm\nh,\hg)\,$ is locally homogeneous.}
\begin{theorem}\label{univc}The leaves of\/ $\,\mathcal{D}^\perp\hskip-2pt$ in
the pseu\-do\hs-Riem\-ann\-i\-an universal covering space\/
$\,(\hm\nh,\hg)\,$ of
any compact rank-one ECS manifold are the factor manifolds of a global product
decomposition of\/ $\,\hm\nh$. More precisely, every leaf\/ $\,L\,$ of\/
$\,\mathcal{D}^\perp\hskip-2pt$ in\/ $\,\hm\,$ is simply connected, and\/
$\,\hm\,$ is dif\-feo\-mor\-phic to\/ $\,\bbR\times L$.
\end{theorem}
We prove both theorems in Section~\ref{pt}.

{\bf Acknowledgments:} The authors greatly appreciate suggestions made by the 
anonymous referee, which allowed us to considerably improve the exposition.

{\bf Competing interests:} The authors declare none.

\renewcommand{\thetheorem}{\thesection.\arabic{theorem}}
\section{Outline of the main argument}\label{om} 
\setcounter{equation}{0}
We fix a compact rank-one ECS manifold $\,(M\nh,\g)\,$ of dimension 
$\,n\ge4$, denote by
$\,(\hm\nh,\hg)\,$ its pseu\-do\hs-Riem\-ann\-i\-an universal covering, and
by $\,\pi:\hm\to M=\hm\nnh/\hh\Gm$ the covering projection. Here 
$\,\Gm\nh\,\approx\,\pi\nh_1\w\hn M\,$ is a group of isometries of 
$\,(\hm\nh,\hg)\,$ acting on $\,\hm\,$ freely and properly
dis\-con\-tin\-u\-ous\-ly. Also, $\,\mathcal{D}\,$ stands for
the Olszak distribution 
(see the Introduction), with the same symbols $\,\mathcal{D}\,$ and
$\,\mathcal{V}\nh=\mathcal{D}^\perp$ denoting objects in $\,\hm\,$ and their
projections onto $\,M\nh$. According to (\ref{wnt}) -- (\ref{flt}), on
$\,\hm\,$ there exists a function $\,t\,$ with a parallel gradient
$\,\nabla\hn t\,$ spanning 
$\,\mathcal{D}$, so that $\,\mathcal{D}^\perp\nnh=\hs\mathrm{Ker}\,dt$, and
the Ric\-ci tensor of $\,(\hm\nh,\hg)$ equals $\,(2-n)\fh(t)\,dt\otimes dt$,
where $\,\fh:\hm\to\bbR\,$ is locally a function of $\,t$. If a $\,C^1$
function $\,\chi:\hm\to\bbR\,$ is locally a function of $\,t$, one may define
its $\,t$-de\-riv\-a\-tive $\,\dot\chi:\hm\to\bbR\,$ by
$\,d\chi=\dot\chi\nnh\,dt$. It easily follows -- cf.\ (\ref{fcg}) -- that
\begin{equation}\label{prj}
\begin{array}{l}
\mathrm{the\ action\ of\ }\,\Gm\hs\mathrm{\ multiplies\ 
}\,\nabla\hn t\,\mathrm{\
by\ nonzero\ constants,}\\
\mathrm{implying\ }\,\Gm\nh\hyp\mathrm{in\-var\-i\-ance\ of\ both\ 
}\,|f|^{1/2}dt\,\mathrm{\ and\ }\,|\dot f|^{1/3}dt\hn.
\end{array}
\end{equation}
We now assume transversal orientability of $\,\mathcal{V}\nh$, and proceed to
summarize the steps leading to the main conclusion of Theorem~\ref{maith}: 
that, unless $\,\g\,$ is locally homogeneous, 
$\,\mathcal{V}\nh=\mathcal{D}^\perp$ must be the vertical 
distribution of a fibration $\,M\nh\to S^1\nh$.

This is achieved by showing that $\,\mathcal{V}\nh$, in addition to being a 
trans\-ver\-sal\-ly-o\-ri\-ent\-able co\-di\-men\-sion-one foliation on the 
compact manifold $\,M\nh$, also has what we call property (\ref{cpl}): {\it
for every compact leaf\/ $\,L\,$ of\/ $\,\mathcal{V}\nh$, the nearby leaves 
are either all noncompact -- with the exception of\/ $\,L\,$ -- or they are
all compact and there exists a prod\-uct-like\/
$\,\mathcal{V}\nh$-sat\-u\-rat\-ed tubular neighborhood of\/ $\,L\,$ in\/}
$\,M\nh$. Furthermore, {\it some compact leaf of\/ $\,\mathcal{V}\hs$ then
realizes the second case in the ei\-ther-or clause of\/} (\ref{cpl}).

The reason why the main claim in Theorem~\ref{maith} follows from the two
conditions italicized above is that, even within the general context of
foliations, with no reference to ECS geometry, these conditions imply that all
leaves of $\,\mathcal{V}\hs$ are compact, which forces $\,M\,$ to be a bundle
over the circle (Theorem~\ref{bdcir}).

Returning to our ECS case, we prove property (\ref{cpl}) for
$\,\mathcal{V}\nh=\mathcal{D}^\perp\nnh$ by first observing, in
Section~\ref{hc},
that the rank-one Ol\-szak distribution $\,\mathcal{D}\,$ on
$\,\hm$ is spanned by the parallel gradient $\,\nabla\hn t$, and so the 
Le\-vi-Ci\-vi\-ta connection of the compact ECS manifold $\,(M\nh,\g)\,$
induces flat linear connections both in $\,\mathcal{D}\,$ (over $\,M$) and in
the line bundle $\,\dla$ over any leaf $\,L\,$ of $\,\mathcal{D}^\perp\nnh$,
dual to the line bundle $\,\mathcal{D}\hskip-2pt_L\w$ arising as the
restriction of $\,\mathcal{D}\,$ to $\,L$. At the same time,
$\,\dla$ is canonically isomorphic to the normal bundle of $\,L\,$ in
$\,M\nh$. Assuming compactness of $\,L\,$ we then show, in
Theorem~\ref{hlncl} (see the next paragraph) that, under a suitable
dif\-feo\-mor\-phic identification $\,\varPsi\,$ 
of a neighborhood $\,\,U\hs$ of $\,L\,$ in $\,M\,$ with a neighborhood 
$\,\,U'$ of the zero section $\,L\,$ in the line bundle $\,\dla$, the
distribution $\,\mathcal{D}^\perp$ on $\,\,U\hs$ corresponds to the
restriction to $\,\,U'$ of the horizontal distribution of the flat linear
connection in $\,\dla$, mentioned above. By (\ref{prj}), the holonomy group
$\,H\hskip-2.5pt_L\w$ of the latter connection consists of multiplications by
positive real constants (`nonzero' in (\ref{prj}) becoming `positive' due to 
transversal orientability of $\,\mathcal{V}\nh=\mathcal{D}^\perp\nnh$), and the
dichotomy required in (\ref{cpl}) comes from the obvious fact that
$\,H\hskip-2.5pt_L\w$ is either infinite, or trivial.

The identification $\,\varPsi:U\to U'\nh$, given by formula (\ref{ptx}), uses
a fixed smooth vector field on $\,M\nh$, nowhere tangent to
$\,\mathcal{D}^\perp\nnh$, to provide the curve segments forming the fibres
of the tubular neighborhood $\,\,U\nh$, and along these segments, pulled
back to $\,\hm\nh$, we define $\,\varPsi\,$ so that it sends local
$\,t$-levels to local sections parallel relative to the flat linear
connection. Due to (\ref{prj}), this construction is
$\,\Gm\nh$-equi\-var\-i\-ant, and hence projects into $\,M\nh$.

Finally, still in the ECS case, we establish the second option of property 
(\ref{cpl}) for some compact leaf of $\,\mathcal{V}\nh=\mathcal{D}^\perp$ by 
considering, in Section~\ref{fs}, the vector space $\,\mathcal{F}$ of all 
continuous functions 
$\,\chi:\hm\to\bbR\,$ such that the $\,1$-form $\,\chi\nnh\,dt\,$ is closed
(i.e., locally exact) and pro\-ject\-a\-ble onto $\,M\nh$, with the linear
operator $\,P:\mathcal{F}\to H^1\nh(M\nh,\hs\bbR)$ sending $\,\chi\,$
to the co\-ho\-mol\-o\-gy class of the projected $\,1$-form on $\,M\nh$.

First, let $\,\dim\mathcal{F}\nh=m<\infty$. By (\ref{prj}),
$\,|f|^{1/2},|\dot f|^{1/3}\nh\in\mathcal{F}\,$ and $\,\mathcal{F}\,$ is 
closed under the $\,m$-ar\-gu\-ment operation assigning
$\,|\psi_1\w\dots\psi_m\w|^{1/m}$ to $\,\psi_1\w,\dots,\psi_m\w$. Simple 
set-the\-o\-ret\-i\-cal reasons then cause $\,|\dot f|^{1/3}$ to be a constant
multiple of $\,|f|^{1/2}$ (see the proof of Theorem~\ref{admts}). This makes
$\,f\,$ globally a function of $\,t$, of the form
$\,f\nh=\ve\hh(t-b)^{-\nh2}$ with real constants $\,\ve\ne0\,$ and $\,b\,$
which, combined with a result from algebraic geometry -- Whitney's theorem
-- implies local homogeneity of $\,\g$ (a precise cross-ref\-er\-ence being: 
Lemma~\ref{babqa} invoked in the proof of Theorem~\ref{lchom}).

On the other hand, when $\,\mathcal{F}\,$ is in\-fi\-nite-di\-men\-sion\-al,
$\,P\,$ must be noninjective due to compactness of $\,M\,$ and, given any
$\,\chi\in\mathcal{F}\smallsetminus\{0\}\,$ with $\,P\chi=0$, we see that 
$\,\chi\nnh\,dt$ projects onto an exact $\,1$-form on $\,M\nh$, and hence
onto $\,d\mu\,$ for some (nonconstant) $\,C^1$ function $\,\mu:M\to\bbR$. As 
$\,\mathcal{D}^\perp\nnh=\hs\mathrm{Ker}\,dt\,$ on $\,\hm\nh$, this $\,\mu\,$
is constant along $\,\mathcal{D}^\perp\nnh$.

Even though $\,\mu\,$ is only guaranteed to be of class $\,C^1\nh$, Sard's
theorem still applies (Remark~\ref{sards}), and any connected component
$\,L\,$ of a regular level of $\,\mu$ clearly realizes the second case of
(\ref{cpl}).

\section{Preliminaries}\label{pr} 
\setcounter{equation}{0}
Manifolds are (usually) connected, pseu\-\hbox{do\hskip.7pt-}Riem\-ann\-i\-an
metrics and vector fields are assumed
$\,C^\infty\nnh$-dif\-fer\-en\-ti\-a\-ble, while
functions may be of lower regularity. The terms `\hs foliation' and 
(in\-te\-gra\-ble) `distribution' will be used interchangeably; by their
`leaves' we always mean {\it maximal connected integral manifolds}.

The following four facts will be used in Sections~\ref{co},~\ref{lh} 
and~\ref{es}.
\begin{remark}\label{flmap}Let points $\,x,x\hh'$ of co\-di\-men\-sion-one 
sub\-man\-i\-folds $\,L,L\nnh'$ of a manifold $\,M\,$ be joined by an integral 
curve $\,\cx$ of a complete $\,C^\infty\nnh$ vector field $\,v\,$ on
$\,M$ which is transverse to $\,L\,$ at $\,x\,$ and to $\,L\nnh'$ at
$\,x\hh'\nh$. Then $\,\cx$ belongs to a smooth variation
$\,\,\widetilde U\nh\ni y\mapsto\cy$ of integral curves of $\,v$,
pa\-ram\-e\-trized by a neighborhood $\,\,\widetilde U$ of $\,x\,$ in
$\,M\,$ such that, for some neighborhoods $\,\,U\nh,U'$ of $\,x\,$ in $\,L\,$ 
and $\,x\hh'$ in $\,L\nnh'$ with $\,\,U\subseteq\widetilde U\nh$, each
$\,\cy$ joins $\,y\in\widetilde U\hs$ to a point $\,y'\nh\in U'$ and
the resulting mapping
$\,y\mapsto y'$ is a sub\-mer\-sion $\,\widetilde U\nh\to U'\nh$, while its
restriction to $\,U\hs$ is a dif\-feo\-mor\-phism $\,\,U\nh\to U'\nh$. The 
word `smooth' also applies here to the domain intervals of the integral curves.

Namely, let $\,\bbR\times M\ni(\taw,y)\mapsto\phi\hs(\taw,y)\in M\,$ be the
flow of $\,v$, so that $\,\cx$ is pa\-ram\-e\-trized by
$\,[\hs0,\taw\hn_*\w]\ni\taw\mapsto\phi\hs(\taw,x)$. We now fix a neighborhood
$\,\,\widetilde U'$ of $\,x\hh'$ in $\,M\,$ and a $\,C^\infty\nnh$ function
$\,\theta:\widetilde U'\nh\to\bbR\,$ with
$\,L\nnh'\nh\cap\hs\widetilde U'\nh=\theta\hh^{-\nnh1}(0)$, having $\,0\,$ as
a regular value. The equation $\,\theta(\phi\hs(\taw,y))=0$, imposed on
$\,(\taw,y)\in\bbR\times\widetilde U\nh$,  
is satisfied when $\,(\taw,y)=(\taw\hn_*\w,x)$. 
The implicit function theorem \cite[p.\ 18]{lang} applied to this equation 
yields a $\,C^\infty\nnh$ function $\,y\mapsto\taw(y)$, defined near 
$\,x\,$ in $\,M\nh$, with $\,\theta(\phi\hs(\taw(y),y))=0$ and 
$\,\taw(x)=\taw\hn_*\w$. Consequently, $\,y'\nh=\phi\hs(\taw(y),y)\,$ lies
in $\,L\nnh'$. The sub\-mer\-sion/dif\-feo\-mor\-phism property of
the mapping $\,\widetilde U\nh\to U'$ or $\,\,U\nh\to U'\nh$ arising in this
way, with $\,\,U\nh=L\cap\hs\widetilde U\nh$, is immediate since
$\,\widetilde U\nh\to U'$ has a right inverse $\,\,U'\nh\to U\hs$ obtained by
using the same principle for $\,-v\,$ instead of $\,v$.
\end{remark}
\begin{remark}\label{flows}For a (possibly disconnected) 
co\-di\-men\-sion-one sub\-man\-i\-fold
$\,L$ of a manifold $\,M\,$ and the flow
$\,\bbR\times M\ni(\taw,x)\mapsto\phi\hs(\taw,x)\in M\,$ of a complete
$\,C^\infty\nnh$ vector field on $\,M\,$ which is nowhere tangent to $\,L$,
the restriction \hbox{$\,\phi:\bbR\times L\to M$} is locally
dif\-feo\-mor\-phic (a co\-di\-men\-sion-zero immersion). If, in addition, 
$\,L\,$ is also compact,
$\,\phi:(-\ve,\ve)\times L\to M$ is an embedding for all $\,\ve>0\,$ close
to $\,0$.

In fact, the first claim follows, since $\,\phi:\bbR\times L\to M\,$ has smooth
local inverses as an easy consequence of Remark~\ref{flmap}. 
Next, if the embedding assertion 
failed, there would exist two term\-wise distinct sequences in
$\,\bbR\times L\,$ having the $\,\bbR\,$ components tending to $\,0$,
with the same sequence of $\,\phi\hskip1.2pt$-\hn val\-ues. Since $\,L\,$ is
now compact,
so are $\,\phi\hs([-\ve,\ve]\times L)\,$ for all $\,\ve>0$, and hence
both sequences have sub\-se\-quences converging to the same limit $\,x\in L$. 
Injectivity of our mapping on a neighborhood of $\,(0,x)$ now leads to a
contradiction.
\end{remark}
By an $\,m$-ar\-gu\-ment operation $\,\varPi\,$ in the following lemma we
mean any mapping associating, with any ordered $\,m$-tuple 
$\,(\psi_1\w,\dots,\psi_m\w)\,$ of functions 
$\,X\to\bbR$, a function $\,\varPi(\psi_1\w,\dots,\psi_m\w):X\to\bbR$. 
\begin{lemma}\label{basis}Let a vector space\/ $\,\mathcal{F}$ of functions\/
$\,X\to\bbR\,$ on a set $\,X$ have a positive finite dimension
$\,m$. If\/ $\,X$ is closed both under the ab\-so\-lute-val\-ue operation\/
$\,\psi\mapsto|\psi|\,$ and under some\/ $\,m$-ar\-gu\-ment operation\/
$\,\varPi$ sending any\/ $\,\psi_1\w,\dots,\psi_m\w$
to a nonnegative function\/
$\,\varPi(\psi_1\w,\dots,\psi_m\w)$ which has the same 
zeros as the product $\,\psi_1\w\nh\ldots\hs\psi_m\w$,
then there exist a proper subset\/ $\,X\nh_0\w$ of\/ $\,X\nnh$, a
partition\/ $\,\{X\hskip-2pt_j\w\}_{j=1}^m$ of\/
$\,X\nh\smallsetminus X\nh_0\w$ and a basis\/
$\,\chi_1\w,\dots,\chi_m\w$ of\/ $\,\mathcal{F}\,$ such
that\/ $\,\chi\nnh_j\w>0\,$ on\/
$\,X\hskip-2pt_j\w$ and\/ $\,\chi\nnh_j\w=0\,$ on\/
$\,X\nh\smallsetminus X\hskip-2pt_j\w$ for
every\/ $\,j=1,\dots,m$. In other words, some basis\/ 
$\,\chi_1\w,\dots,\chi_m\w$ of\/ $\,\mathcal{F}\,$ consists of nonnegative
functions with pairwise disjoint supports, the word `support\hn' meaning here
the complement of the zero set.  
\end{lemma}
\begin{proof}Choose $\,x_1\w,\dots,x_m\w\in X\,$ with the $\,m\,$ evaluations 
$\,\delta_{x\nnh_j}$ forming  a basis of the dual space 
$\,\mathcal{F}^*\nh$, so that some
$\,\sigma\nnh_1\w,\dots,\sigma\nnh_m\w\in\mathcal{F}\,$ have 
$\,\sigma\hskip-2pt_j\w(x\nh_i\w)<0<\sigma\hskip-2pt_j\w(x\nnh_j\w)$ whenever
$\,i\ne j$. 
The pos\-i\-tive-part and neg\-a\-tive-part operations 
$\,(\hskip2.3pt)^\pm:\mathcal{F}\to\mathcal{F}$ are given by 
$\,\sigma^\pm\hskip-.7pt=(|\hh\sigma\nh|\pm\sigma\nh)/2$. We set 
$\,\chi\nnh_j\w=
\varPi(\widetilde\sigma\hskip-2pt_1\w,\dots,\widetilde\sigma\hskip-2pt_m\w)$, where 
$\,\widetilde\sigma\hskip-2pt_j\w=\sigma\hskip-2pt_j^+$ and 
$\,\widetilde\sigma\nnh_i\w=\sigma\nnh_i^-$ for $\,i\ne j$. The
supports 
$\,X\hskip-2pt_j\w=\chi\nnh_j^{-\nnh1}((0,\infty))\,$ are nonempty (as
$\,x\nnh_j\w\in X\hskip-2pt_j\w$), and pairwise disjoint (since, if
$\,i\ne j$, one has $\,\sigma\nnh_i\w<0<\sigma\hskip-2pt_j\w$ on 
$\,X\hskip-2pt_j\w$), which trivially implies linear independence of
$\,\chi_1\w,\dots,\chi_m\w$. Our claim thus follows
if we declare $\,X\nh_0\w$ to be
$\,X\nnh\smallsetminus\bigcup_{j=1}^mX\hskip-2pt_j\w$, that is, the
simultaneous zero set of $\,\chi_1\w,\dots,\chi_m\w$.
\end{proof}
\begin{remark}\label{lcdif}A lo\-cal\-ly-dif\-feo\-mor\-phic $\,C^\infty\nnh$
mapping from a compact manifold into a connected one is necessarily
surjective, and so the latter manifold must be compact as well. 
In fact, the 
image of the mapping is both compact and open. (We do not need the well-known 
stronger conclusion that the mapping is then also a covering projection.)
\end{remark}
\begin{remark}\label{ccomp}A nonempty proper sub\-set $\,K\, $ of
$\,(0,\infty)$, different from $\,\{1\}$, and closed under the mappings
$\,q\mapsto q^r$ for all $\,r\in\bbZ$, must have infinitely many connected 
components: if $\,q\in(1,\infty)\smallsetminus K\,$ is fixed, choosing a
connected component $\,K\nh_r\w$ of each nonempty intersection
$\,K\cap(q^r\nh,q^{r+1})$, $\,r\in\bbZ$, we obtain, due to
unboundedness of $\,K\nnh$, an infinite family of such components
$\,K\nh_r\w$.
\end{remark}

\section{Co\-di\-men\-sion-one foliations}\label{co}
\setcounter{equation}{0}
The results of this section are the most crucial steps in proving
Theorem~\ref{maith}.

Here is a property of a co\-di\-men\-sion-one foliation 
$\,\mathcal{V}\hs$ on a manifold $\,M$:
\begin{equation}\label{cpl}
\begin{array}{l}
\mathrm{every\nh\ compact\nh\ leaf\nh\ }\,\hs L\,\mathrm{\nh\ of\nh\
}\,\mathcal{V}\,\mathrm{\nh\ has\nh\ a\nh\ neighborhood\nh\
}\,\hs U\mathrm{\nh\ in\nh\ }\,M\mathrm{\nh\ such\nh\ that}\\
\mathrm{the\hh\ leaves\hh\ of\hh\ }\hs\mathcal{V}\hs\mathrm{\hh\
intersecting\hh\ }\hs\,U\nnh\smallsetminus\hh L\,\mathrm{\hh\ are\hh\
either\hh\
all\hh\ noncompact,\nnh\ or}\\
\mathrm{they\hn\ are\hn\ all\hn\ compact\hn\ and\hn\ some\hn\ neighborhood\hn\
of\hn\ }\hs L\mathrm{\hh\ in\hn\ }\hs M\mathrm{\hn\ is\hn\ a\hn\ union}\\
\mathrm{of\hs\ compact\hs\ leaves\hs\ of\hs\ }\,\,\mathcal{V}\,\hs\mathrm{\
and\hs\ may\hs\ be\hs\ dif\-feo\-mor\-phic\-al\-ly\hs\ identified}\\
\mathrm{with\hs\ }\,\hs\bbR\nh\times L\,\,\mathrm{\ so\hs\ as\hs\ to\hs\ make\
}\,\,\mathcal{V}\,\hs\mathrm{\hs\ appear\hs\ as\hs\ the\hs\
}\,\,L\,\,\mathrm{\hs\ factor\hs\ foliation.}
\end{array}
\end{equation}
The final clause of (\ref{cpl}) means that some dif\-feo\-mor\-phism of 
a neighborhood of $\,L$ in $\,M\,$ onto $\,\bbR\nh\times L\,$ pushes
$\,\mathcal{V}\,$ forward onto 
the $\,L\,$ {\it factor foliation\/} of $\,\bbR\nh\times L$, which has the
leaves $\,\{\taw\}\times L$, for $\,\taw\in\bbR$.
\begin{theorem}\label{bdcir}Let a trans\-ver\-sal\-ly-o\-ri\-ent\-able
co\-di\-men\-sion-one foliation\/ $\,\mathcal{V}\hs$ on a compact manifold\/
$\,M\hs$ satisfy condition\/ {\rm(\ref{cpl})}. If, in addition, some compact 
leaf\/ $\,L\,$ of $\,\mathcal{V}\hs$ realizes the second possibility in\/
{\rm(\ref{cpl})}, so as to have a prod\-uct-like\/
$\,\mathcal{V}\nh$-sat\-u\-rat\-ed neighborhood in $\,M\hs$ formed by compact 
leaves, then the leaves of\/ $\,\mathcal{V}\hs$ are all compact and they
constitute the fibres of a bundle projection\/ $\,M\nh\to S^1\nnh$.
\end{theorem}
\begin{proof}Transversal orientability of $\,\mathcal{V}\hs$ allows us to fix
a $\,C^\infty\nnh$ vector field $\,v$ on $\,M\nh$, nowhere tangent to
$\,\mathcal{V}\nh$. We also fix a compact leaf $\,L\,$ of $\,\mathcal{V}\hs$
satisfying the second option in the ei\-ther-or clause of (\ref{cpl}). With 
$\,\bbR\times M\ni(\taw,x)\mapsto\phi\hs(\taw,x)\in M$ denoting the flow
of $\,v$, and $\,z\,$ a given point of $\,L$, the ``second option'' guarantees
the existence of an open interval $\,(a'\nh,b')\,$ containing $\,0\,$ such
that, for all $\,\taw\in(a'\nh,b')$, the leaf $\,L\nh_{\phi(\taw,z)}\w$ of
$\,\mathcal{V}\hs$ passing through $\,\phi\hs(\taw,z)\,$ is compact. Let
$\,(a,b)\,$ be the maximal open interval with this property (that is, the
union of all such intervals).

All $\,L\nh_{\phi(\taw,z)}\w$ with $\,\taw\in(a,b)\,$ satisfy
the second option in (\ref{cpl}), as their compactness obviously precludes the
first one. The resulting product structure of a neighborhood of each 
$\,L\nh_{\phi(\taw,z)}\w$ has three immediate consequences. First, the set 
$\,E=\{(\taw,y)\in(a,b)\times M:y\in L\hn_{\phi(\taw,z)}\w\}\,$ is open in
$\,(a,b)\times M\nh$. Secondly, the mapping 
$\,E\ni(\taw,y)\mapsto\taw\in(a,b)\,$ constitutes a bundle projection; and, 
finally,
\begin{equation}\label{lcd}
\mathrm{the\ mapping\ }\,E\ni(\taw,y)\mapsto y\in M\,\mathrm{\ is\
lo\-cal\-ly\ dif\-feo\-mor\-phic.}
\end{equation}
The pull\-back of $\,v\,$ under this last mapping is a vector field on the 
total space $\,E\,$ of the bundle in (b), transverse to the fibres, so that 
it spans a nonlinear connection (horizontal distribution) in $\,E$.

For any fixed $\,x\in L$, let $\,(a,b)\ni\taw\mapsto(\taw,\lambda(\taw,x))
\in E\,$ be the horizontal lift, relative to this connection, of the curve 
$\,\taw\mapsto\taw\,$ in the base manifold $\,(a,b)$, with the initial value 
$\,(0,x)\,$ at $\,\taw=0$. Such a horizontal lift clearly exists on some 
neighborhood $\,(c,d)\,$ of $\,0\,$ in $\,(a,b)$, as it constitutes a 
solution to an ordinary differential equation. Compactness of the fibres 
$\,E\nh_{\taw}\w=\{\taw\}\times L\hn_{\phi(\taw,z)}\w$ guarantees in turn 
that the maximal such $\,(c,d)\,$ equals $\,(a,b)$. Namely, if we had $\,c>a$,
or $\,d<b$, a neighborhood of $\,E\nh_c\w$ or $\,E\nh_d\w$ in $\,E\,$ forming
a union of horizontal curves would have the property that a horizontal lift 
entering it can be extended so as to reach $\,E\nh_c\w$ or $\,E\nh_d\w$ and 
beyond, contrary to maximality of $\,(c,d)$.

Note that $\,(a,b)\ni\taw\mapsto\lambda(\taw,x)$, for any given $\,x\in L$,
is a re\-pa\-ram\-e\-tri\-za\-tion of the integral curve 
$\,\taw\mapsto\phi\hs(\taw,x)\,$ of $\,v$. Thus, for any
$\,(\taw,x)\in(a,b)\times L$,
\begin{equation}\label{cxe}
\cx=\{\lambda(\taw,x):\taw\in(a,b)\}\,\mathrm{\ is\ an\ 
un\-pa\-ram\-e\-triz\-ed\ in\-te\-gral\ curve\ of\ }\,v
\end{equation}
passing through $\,x\,$ for $\,\taw=0\,$ and, for some $\,C^\infty\nnh$ 
function $\,\sigma:(a,b)\times L\to\bbR$,
\begin{equation}\label{ltx}
\lambda(\taw,x)\,=\,\phi\hs(\sigma(\taw,x),x)\hh,\hskip12pt\lambda(\taw,z)\,
=\,\phi(\taw,z) 
\end{equation}
whenever $\,(\taw,x)\in(a,b)\times L$. Also, 
\begin{equation}\label{sgm}
\sigma(0,x)=x\hh,\hskip5pt
\sigma(\taw,z)=\taw\hh,\hskip5pt
\phi\hs(\sigma(\taw,x),x)\nh\in\nh L_{\phi(\taw,z)}\w\hh,\hskip5pt
d\hs[\sigma(\taw,x)]/d\taw>0\hh,
\end{equation}
$d\hh[\sigma(\taw,x)]/d\taw\,$ being positive as it is nonzero everywhere and
$\,\sigma(\taw,z)=\taw$.

We now proceed to establish the following conclusion:
\begin{equation}\label{ecd}
\mathrm{there\ exist\ }\,c,d\,\mathrm{\ with\ }\,a<c<d<b\,\mathrm{\ and\
}\,L\nh_{\phi(c,z)}\w=L\hn_{\phi(d,z)}\w.
\end{equation}
We will achieve this by deriving a contradition from the assumption that
\begin{equation}\label{mtd}
L\nh_{\phi(\taw,z)}\w\mathrm{,\ for\ }\,\taw\in(a,b)\mathrm{,\ are\ all\
mutually\ distinct.}
\end{equation}
First, $\,\lambda(\taw,x)\in L_{\phi(\taw,z)}\w$ by (\ref{ltx}) --
(\ref{sgm}), so that (\ref{mtd}) and (\ref{cxe}) give
\begin{equation}\label{sgp}
\cx\cap L_{\phi(\taw,z)}\w=\{\lambda(\taw,x)\}\,\mathrm{\ for\ any\
}\,(\taw,x)\in(a,b)\times L\hh.
\end{equation}
Positivity of $\,d\hh[\sigma(\taw,x)]/d\taw\,$ -- see (\ref{sgm}) -- implies,
whenever $\,x\in L$, that $\,\sigma(\taw,x)\,$ has a limit
$\,\sigma(b,x)\le\infty\,$ as $\,\taw\to b$. As a further consequence of
(\ref{mtd}), $\,\sigma(b,x)<\infty$ for every $\,x\in L\,$ (and so, by
(\ref{sgm}), $\,b=\sigma(b,z)<\infty$). Otherwise, we may fix $\,x\in L\,$
and a strictly increasing sequence 
$\,\taw\nnh_j\w>a\,$ with $\,\taw\nnh_j\w\to b\,$ and 
$\,\sigma(\taw\nnh_j\w,x)\to\infty$ as $\,j\to\infty$. The sequence 
$\,\lambda(\taw\nnh_j\w,x)=\phi\hs(\sigma(\taw\nnh_j\w,x),x)$ lies in the
single integral curve $\,\cx$ and, passing to a sub\-se\-quence,
we may assume that it converges to some point $\,y\in M\nh$, which has a 
neighborhood in $\,M\,$ forming a union of ``short'' un\-pa\-ram\-e\-triz\-ed 
segments of in\-te\-gral curves of $\,v\,$ intersecting a neighborhood of
$\,y$ in the leaf $\,L\nh_y\w$. Since $\,\lambda(\taw\nnh_j\w,x)\to y\,$ while
its parameter $\,\sigma(\taw\nnh_j\w,x)\,$ increases towards an infinite 
limit and, due to (\ref{mtd}) -- (\ref{sgp}),
$\,(a,b)\ni\taw\mapsto\lambda(\taw,x)\,$ is injective, 
$\,\cx$ must contain infinitely many of the ``short'' in\-te\-gral-curve 
segments and, as a result, intersect some
leaves $\,L_{\phi(\taw,z)}\w$, with $\,\taw=\taw\nnh_j\w$, at infinitely
many points, contrary to (\ref{sgp}). Thus, $\,\sigma(b,x)<\infty\,$ whenever
$\,x\in L$.

Next, let us set $\,\lambda(b,x)=\phi\hs(\sigma(b,x),x)\,$ and 
denote by $\,L_{b,x}\w=L\nh_{\lambda(b,x)}\w$ the leaf of $\,\mathcal{V}\hs$
through $\,\lambda(b,x)$. Now
\begin{equation}\label{lcc}
\mathrm{the\ mapping\ }\,L\ni x\mapsto L\nh_{b,x}\w\mathrm{\ is\ locally\
constant,}
\end{equation}
and $\,L\ni x\mapsto\sigma(b,x)\,$ is
$\,C^\infty\nnh$-dif\-fer\-en\-ti\-a\-ble. In fact, given $\,x\in L$, the
integral curve (\ref{cxe}) joins $\,x\,$ to 
$\,\lambda(b,x)\in L_{b,x}\w$. Remark~\ref{flmap} applied to
our $\hs L\hs$ and $\,L\nnh'\nh=L_{b,x}\w$ yields a smooth variation of
integral curves $\,\cy\,$ of $\,v$, for all points $\,y\,$ from
a neighborhood $\,\,U\,$ of $\,x\,$ in $\,L$, with each $\,\cy$ joining
$\,y\,$ to a point in $\,L_{b,x}\w$. Each of these $\,\cy$ is
pa\-ram\-e\-trized by $\,\taw\mapsto\phi\hs(\taw,x)\,$ with $\,\taw\,$ ranging
over an interval $\,[\hs0,\taw\hn_*\w]$, where $\,\taw\hn_*\w$ depends on
$\,y$. As before, $\,\lambda(b,x)\,$ has a neighborhood in $\,M\,$
constituting a union of ``short'' un\-pa\-ram\-e\-triz\-ed
in\-te\-gral-curve segments intersecting a neighborhood of 
$\,\lambda(b,x)\,$ in $\,L_{b,x}\w$. One of these segments, the one 
passing through $\,\lambda(b,x)$, contains the portion 
$\,\{\lambda(\taw,x):b-\ve\le\taw\le b\}\,$ of $\,\cx$, with some $\,\ve>0$.
The third equality of (\ref{sgm}) along with (\ref{sgp}) for $\,y\in L\,$ near 
$\,x\,$ (rather than $\,x\,$ itself) show that each nearby $\,\cy$ similarly
contains 
$\,\{\lambda(\taw,y):b-\ve\le\taw<b\}$ and hence also the limit
$\,\lambda(b,x)$. However, by (\ref{sgp}), all $\,\lambda(\taw,y)\,$ lie 
in $\,L\nh_{\phi(\taw,z)}\w$, just as $\,\lambda(\taw,x)\,$ does,
and so $\,\lambda(b,y)\in L_{b,x}\w$, which gives  
$\,L_{b,y}\w=L_{b,x}\w$ and thus proves (\ref{lcc}). At the same time, 
Remark~\ref{flmap} yields smoothness of the mapping
$\,L\ni y\mapsto\sigma(b,y)$.

Since the leaf $\,L\,$ is connected, local constancy in (\ref{lcc}) amounts 
to constancy, so that
$\,L\nh_{\lambda(b,x)}\w=L\nh_{\phi\hs(b,z)}\w$ for all $\,x\in L$.
Thus, for every $\,x\in L$, the integral curve  
$\,[\hs0,\sigma(b,x)]\ni\taw\mapsto\phi\hs(\taw,x)\in M\,$ joins 
$\,x\,$ to the point $\,\lambda(b,x)\,$ in 
$\,L\nh_{\phi\hs(b,z)}\w$. 
Remark~\ref{flmap} also implies that the mapping 
$\,L\ni x\mapsto\lambda(b,x)\in L\nh_{\phi\hs(b,z)}\w$ is locally
dif\-feo\-mor\-phic. Compactness of $\,L\hs$ and Remark~\ref{lcdif} now yield 
compactness of $\,L\nh_{\phi\hs(b,z)}\w$ along with the second possibility
in (\ref{cpl}), for $\,L\nh_{\phi\hs(b,z)}\w$, since the first one is
precluded by compactness of the nearby leaves $\,L\nh_{\phi\hs(\taw,z)}\w$
with $\,a<\taw<b$. This in turn also implies
compactness of $\,L\nh_{\phi\hs(\taw,z)}\w$ for $\,\taw\ge b$, close to
$\,b$, contradicting maximality of $\,(a,b)$ and, consequently,
proving (\ref{ecd}).

Let us now fix $\,c,d\,$ with (\ref{ecd}). The image, under the mapping in
(\ref{lcd}), of the set $\,E_{[\hh c,d\hs]}\w
=\{(\taw,y)\in[\hh c,d\hs]\times M:y\in L\hn_{\phi(\taw,z)}\w\}\,$  is then clearly compact, but also open in $\,M\nh$. In fact, it
suffices to verify that $\,L\nh_{\phi(c,z)}\w$ is contained in the interior of
this image -- which trivially follows since the image contains both kinds of
sufficiently small one-sided neighborhoods of 
$\,L\nh_{\phi(c,z)}\w\nh=L\hn_{\phi(d,z)}\w$ in
$\,M\nh$. The image thus coincides with $\,M\nh$, which proves that
all leaves of $\,\mathcal{V}\hs$ are compact.

Therefore, there exists a bundle projection $\,M\nh\to S^1$ with the
leaves of $\,\mathcal{V}$ serving as the fibres 
\cite[Exercise 2.29(3)(i) on p.\ 49]{moerdijk-mrcun}.
\end{proof}

\section{Rank-one ECS metrics}\label{ro} 
\setcounter{equation}{0}
Let the data $\,\fh\nh,I\nh,n,\mv\nh,\lr,A\,$ consist of
\begin{equation}\label{dta}
\begin{array}{l}
\mathrm{a\ nonconstant\ }\,\,C^\infty\,\mathrm{\ function\
}\,\,\,\fh:I\hs\to\hh\bbR\hs\,\,\mathrm{\ on\ an\ open\ interval}\\
I\hskip1.5pt\subseteq\bbR\mathrm{,\ an\ integer\ }\,n\ge4\mathrm{,\ a\ real\
vector\ space\ }\,\mv\hh\mathrm{\ of\ dimension}\\
n\nh-\nh2\mathrm{,\nnh\ a\nh\ pseu\-do\hs}\hyp\mathrm{Euclid\-e\-an\nh\
inner\nh\ product\ }\nh\lr\hs\mathrm{\nh\ on\nh\ }\,\mv\nnh\nnh\mathrm{,\nh\
and\nh\ a\nh\ non}\hyp\\
\mathrm{zero,\ traceless,\ 
}\nnh\lr\hyp\mathrm{self}\hyp\mathrm{ad\-joint\ linear\
en\-do\-mor\-phism\ }\,A\,\mathrm{\ of\ }\,\hs\mv\nnh\nh.
\end{array}
\end{equation}
Following \cite{roter}, one then defines a rank-one ECS metric
\cite[Theorem~4.1]{derdzinski-roter-09}
\begin{equation}\label{met}
g\,=\,\kappa\,dt^2\nh+\,dt\,ds\hs+\hs\delta
\end{equation}
on the $\,n$-di\-men\-sion\-al manifold $\,I\nh\times\bbR\times\mv\nh$. The
products of differentials denote here symmetric products, $\,t,s\,$ are the
Cartesian coordinates on $\,I\nh\times\bbR\,$ treated, with the aid of the
projection $\,I\nh\times\bbR\times\mv\to I\nh\times\bbR$, as functions 
on $\,I\nh\times\bbR\times\mv\nh$, and $\,\delta\,$ is the pull\-back to
$\,I\nh\times\bbR\times\mv\hs$ of the flat \prc\ on $\,\mv\hs$
corresponding to the inner product $\,\lr$. Finally, 
$\,\kappa:I\nh\times\bbR\times\mv\nh\to\bbR\,$ is the function given by 
\begin{equation}\label{kap}
\kappa(t,s,x)\,=\,f(t)\hskip.4pt\langle x,x\rangle\,+\,\langle Ax,x\rangle\hh.
\end{equation}
If we let $\,i,j\,$ range over $\,2,\dots,n-1$, fix linear 
coordinates $\,x^i$ on $\,\mv$ and use them, along with $\,x^1\nh=t\,$ on
$\,I\hs$ and $\,x^n\nh=s/2\,$ on $\hs\bbR$, to form a global coordinate system
on $\,I\nh\times\bbR\times\mv\nnh$, then the pos\-si\-bly-non\-ze\-ro
components of the metric $\,\g\,$ and the Le\-vi-Ci\-vi\-ta connection
$\,\nabla\hs$ are \cite[p.\ 93]{roter} those algebraically related to
\begin{equation}\label{pnz}
\begin{array}{l}
g_{11}\w=\kx\hh,\hskip17ptg_{1n}\w=g_{n1}\w=1\hh,\hskip17pt
\mathrm{and\ (constants)\ }\,g_{ij}\w\hh,\\
\vg_{\hskip-2.7pt11}^{\hs n}=\partial\nh_1\w\kx/2\hh,\hskip19pt
\vg_{\hskip-2.7pt11}^{\hs i}=-g^{ij}\partial\nnh_j\w\kx/2\hh,\hskip19pt
\vg_{\hskip-2.7pt1i}^{\hs n}=\partial\nh_i\w\kx/2\hh.
\end{array}
\end{equation}
\begin{remark}\label{relax}Conversely \cite[Theorem~4.1]{derdzinski-roter-09},
in any $\,n$-di\-men\-sion\-al rank-one ECS manifold, some neighborhood of
any given point is isometric to an open su\-bset of a manifold of type
(\ref{met}), where one has (\ref{dta}) with one possible exception: $\,f\,$
may be constant. More precisely, $\,df/dt=0\,$ precisely at those points 
at which the covariant derivative $\,\nabla\hskip-1.5ptR\,$ of the curvature
tensor vanishes. Thus, if $\,f\,$ is constant on a sub\-in\-ter\-val $\,I\hn'$
of $\,I\nh$, the metric (\ref{met}) will have $\,\nabla\hskip-1.5ptR=0\,$ on 
$\,I\hn'\nh\times\bbR\times\mv\nh$.

In other words, a rank-one ECS manifold may have locally symmetric open
sub\-man\-i\-folds, and then the function $\,f\,$ appearing in the
lo\-cal-co\-or\-di\-nate form of \cite[Theorem~4.1]{derdzinski-roter-09} 
is constant.
\end{remark}

\section{Assumptions and notation}\label{an}
\setcounter{equation}{0}
Our $\,(M\nh,\g)\,$ is always a rank-one ECS manifold, often compact, and
$\,(\hm\nh,\hg)$ denotes its
pseu\-do\hs-Riem\-ann\-i\-an universal covering space, which leads to
\begin{equation}\label{ucp}
\mathrm{the\ universal\ covering\ projection\
}\,\pi:\hm\to M=\hm\nnh/\hh\Gm\hh,
\end{equation}
$\Gm\nh\,\approx\,\pi\nh_1\w\hn M\,$ being a group of isometries of
$\,(\hm\nh,\hg)\,$ acting on $\,\hm\,$ freely and properly 
dis\-con\-tin\-u\-ous\-ly. Most of the time, the same symbols will stand for 
objects in $\,\hm$ and their projections in $\,M\nh$, such as the metric
$\,\g$, and the (rank-one) Ol\-szak distribution 
$\,\mathcal{D}\,$ described in the Introduction. In any rank-one ECS manifold,
the distribution $\,\mathcal{D}$, being parallel, carries a linear
connection induced by the Le\-vi-Ci\-vi\-ta connection of $\,\g$, and this
induced connection is flat since, locally, 
$\,\mathcal{D}\,$ is spanned by the parallel gradient 
$\,\nabla\hn t\,$ of the coordinate function $\,t\,$ appearing in 
(\ref{met}) -- see \cite[the lines following formula
(3.6)]{derdzinski-terek-ne}. (As stated in Remark~\ref{relax}, (\ref{met}) is a
general local description of all rank-one ECS metrics.) Simple connectivity of
$\,\hm\,$ allows
us to choose a global parallel vector field $\,w\,$ on $\,\hm\hs$ spanning
$\,\mathcal{D}$,
and then the $\,1$-form $\,\g(w,\,\cdot\,)$, being parallel, is closed, and
hence exact, so that we may also fix a $\,C^\infty\nnh$ function
$\,t:\hm\to\bbR\,$
with $\,dt=\g(w,\,\cdot\,)$. This determines $\,t\,$ uniquely up to af\-fine
substitutions, that is, replacements of $\,t\,$ by $\,qt+p\,$ with
$\,q,p\in\bbR\,$ and $\,q>0$. Positivity of $\,q\,$ reflects the fact
that we always assume transversal orientability of
the orthogonal complement $\,\mathcal{D}^\perp\nnh$, 
which can be achieved by 
replacing $\,(M\nh,\g)\,$ with a two-fold isometric covering, and
amounts to requiring that $\,\Gm\,$ be a sub\-group of the group
$\,\mathrm{Iso}\nh^+\nh(\hm\nh,\g)\,$ of all self-isom\-e\-tries of
$\,(\hm\nh,\g)\,$ preserving a fixed transversal orientation of
$\,\mathcal{D}^\perp\nnh$. Thus, for every
$\,\gamma\in\mathrm{Iso}\nh^+\nh(\hm\nh,\g)$,
\begin{equation}\label{tqp}
\mathrm{there\ exist\ unique\ }\,q,p\in\bbR\,\mathrm{\ with\
}\,q>0\,\mathrm{\ and\ }\,t\circ\gamma=qt+p\hh. 
\end{equation}
The assignment $\,\gamma\mapsto q\,$ is a
group homo\-mor\-phism $\,\mathrm{Iso}\nh^+\nh(\hm\nh,\g)\to(0,\infty)$, and
we refer to $\,q\,$ as the {\it$\,q$-im\-a\-ge\/} of $\,\gamma$. 
Summarizing, 
\begin{equation}\label{wnt}
w=\nabla\hn t\,\mathrm{\ is\ a\ parallel\ vector\ field\ on\
}\,\hm\nh\mathrm{,\ spanning\ }\,\mathcal{D}\mathrm{,\ and\
}\,dt=\g(w,\,\cdot\,)\hh.
\end{equation}
As a consequence of (\ref{wnt}), 
\begin{equation}\label{ker}
\mathcal{D}^\perp\nh=\,\hs\mathrm{Ker}\,dt\,\,\mathrm{\ on\ }\,\hs\hm.
\end{equation}
Since (\ref{met}) remains unchanged when $\,t\,$ and 
$\,s\,$ are replaced by $\,qt+p\,$ and $\,q^{-\nnh1}\nh s$,
\begin{equation}\label{our}
t\,\mathrm{\ in\ (\ref{met})\ can\ always\ be\ made\ equal\ to\ our\
}\,t\,\mathrm{\ chosen\ as\ above,}
\end{equation}
cf.\ Remark~\ref{relax}. 
According to \cite[p.\ 93]{roter}, where the curvature tensor has the sign
opposite to ours, the metric (\ref{met}) has the Ric\-ci tensor 
$\,\mathrm{Ric}=(2-n)\fh(t)\,dt\otimes dt$, for $\,n=\dim M\nh$.
Therefore, by Remark~\ref{relax}, with our $\,t\,$ chosen as above,
\begin{equation}\label{ric}
\mathrm{on\ }\,\,(\hm\nh,\g)\hs\,\mathrm{\ one\ has\ }\,\,\mathrm{Ric}\,
=\,(2-n)\fh\hs dt\otimes dt
\end{equation}
for a unique (nonconstant) function $\,f:\hm\to\bbR\,$ and, again by
Remark~\ref{relax}, 
\begin{equation}\label{flt}
f\,\,\mathrm{\ is\ locally\ a\ function\ of\ }\,\,t\hh.
\end{equation}
Changing the notational convention so as to absorb the factor $\,2-n\,$ into
the function $\,f\,$ does not seem to be a good option, since the inverse of
that factor then would have to appear in (\ref{kap}). Also, the word `locally'
in (\ref{flt}) cannot in general be skipped: it means that $\,f\,$ is constant
on the connected components of the level sets of $\,t$, and the level sets
themselves may be disconnected.

If $\,\chi:\hm\to\bbR\,$ is of class $\,C^1\nnh$ and, locally, a
function of $\,t\,$ (which amounts to $\,d\chi\,$ being a functional 
multiple of $\,dt$), we define the derivative $\,\dot\chi:\hm\to\bbR\,$ by 
$\,d\chi=\dot\chi\nnh\,dt$, so that, in terms of a local expression of
$\,\g\,$ in (\ref{met}) -- (\ref{pnz}), $\,\dot\chi=d\chi/dt$. 
For $\,f$ in (\ref{ric}) -- (\ref{flt}), any $\,\gamma\in\Gm\nh$, any
$\,C^1$ function $\,\chi:\hm\to\bbR\,$ which is locally a function of $\,t$,
and $\,q,p,a\in\bbR\,$ with $\,q>0$,
\begin{equation}\label{fcg}
\begin{array}{l}
\mathrm{if\ \ }\,t\circ\gamma\,=\,qt\,+\,p\mathrm{,\ \ then\ \ }\,\gamma^*dt\,
=\,q\,dt\,\mathrm{\ \ and\ \ }\,f\circ\gamma\,=\,q^{-\nh2}f,\\
\mathrm{if\ \ }\,\chi\circ\gamma=q\hh^a\chi\mathrm{,\ \ then\
}\,\dot\chi\circ\gamma=q\hh^{a-1}\chi\hh,
\end{array}
\end{equation}
which is clear from (\ref{ric}) and the fact that the pull\-back of
differential forms commutes with exterior differentiation. By
(\ref{wnt}) and (\ref{fcg}), for any
$\,\gamma\in\mathrm{Iso}\nh^+\nh(\hm\nh,\g)$,
\begin{equation}\label{gpw}
\gamma\,\,\mathrm{\ pulls\ }\,w\,\mathrm{\ back\ to\ }\,qw\mathrm{,\ where\
}\,q\in(0,\infty)\,\mathrm{\ is\ the\ }\,q\hyp\mathrm{im\-a\-ge\ of\
}\,\gamma\hh.
\end{equation}
Let us point out that the choices of $\,t\,$ and $\,w\,$ made above
are convenient, but not canonical, and so, rather than being preserved by
isometries, $\,t\,$ and $\,w\,$ are transformed by them via (\ref{fcg}) and
(\ref{gpw}).

\section{Local homogeneity}\label{lh}
\setcounter{equation}{0}
We adopt here the assumptions and notation of Section~\ref{an}. First, note
that
\begin{equation}\label{rns}
\mathrm{in\ any\ ECS\ manifold,\ }\, \mathrm{Ric}\ne0\,\,\mathrm{\ somewhere,}
\end{equation}
or else the curvature and Weyl tensors would coincide,
implying local symmetry.

Just as was the case with $\,\mathcal{D}$, the distribution
$\,\mathcal{D}^\perp\nnh$ is parallel, and so it inherits a linear 
connection from the Le\-vi-Ci\-vi\-ta connection $\,\nabla\hs$ of $\,\g$. As 
$\,\mathcal{D}\subseteq\mathcal{D}^\perp\nnh$, a natural connection
arises in the quotient bundle
$\,\mathcal{E}=\mathcal{D}^\perp\hskip-2pt/\mathcal{D}\,$
over $\,\hm\nh$, or $\,M\nh$, and
\begin{equation}\label{tcf}
\mathrm{\ the\ connection\ induced\ by\ }\,\nabla\hs\mathrm{\ in\ 
}\,\hs\mathcal{E}=\mathcal{D}^\perp\hskip-2pt/\mathcal{D}\,\mathrm{\ is\ flat.}
\end{equation}
This was established in \cite[Lemma 2.2\hs-f]{derdzinski-roter-09}, but we
need to justify it here, again, to draw additional conclusions. 
Namely, by (\ref{pnz}) and (\ref{wnt}), the coordinate vector field 
$\,\partial\nh_n\w=w=\nabla\hn t$, spanning $\,\mathcal{D}$, is parallel,
while $\,\partial\nh_n\w$ and $\,\partial\nh_i\w$, for $\,i,j=2,\dots,n-1$,
span $\,\mathcal{D}^\perp\nnh$. Now (\ref{tcf}) follows since, in (\ref{pnz}), 
$\,\vg_{\hskip-2.7pt1n}^{\hs\bullet}=\vg_{\hskip-2.7pt1i}^{\hs\bullet}
=\vg_{\hskip-2.7ptin}^{\hs\bullet}=\vg_{\hskip-2.7ptij}^{\hs\bullet}
=\vg_{\hskip-2.7ptnn}^{\hs\bullet}=\vg_{\hskip-2.7ptni}^{\hs\bullet}=0$, 
with $\,\bullet\,$ denoting any index other than $\,n$, so that
$\,\partial\nh_i\w$, $\,i=2,\dots,n-1$, project onto parallel local 
trivializing sections of $\,\mathcal{E}\nh$. 
As $\,\partial\nh_i\w$ are also constant vector fields on the
space $\,\mv$ in (\ref{dta}) -- (\ref{met}), we may identify $\,\mv\hs$
with the space of parallel sections of
$\,\mathcal{E}\nh$, defined on the coordinate domain.
The parallel vec\-tor-bun\-dle morphism
$\,[\mathcal{D}^*]^{\otimes2}
\to\mathcal{E}^{\otimes2}\nnh$, over any rank-one
ECS manifold, defined in \cite[formula (6)]{derdzinski-roter-08} (where it was 
denoted by $\,\varPhi$, and built from the Weyl tensor $\,W\nnh$, cf.\ the
Appendix), is easily seen
to be valued in $\,\mathcal{E}^{\odot2}\nnh$. Let the parallel section
$\,\xi\,$ of the line bundle $\,\mathcal{D}^*$ over $\,\hm\,$ be
\begin{equation}\label{dua}
\mathrm{dual\ to\ }\,w\,\mathrm{\ in\ the\ sense\ that\ }\,\xi(w)=1\hh,
\end{equation}
for the trivializing parallel section $\,w\,$ of $\,\mathcal{D}\,$ appearing
in (\ref{wnt}). Thus, over $\,\hm\nh$, the morphism
$\,[\mathcal{D}^*]^{\otimes2}\to\mathcal{E}^{\otimes2}$ sends
$\,\xi\hn\otimes\hs\xi\,$ to a parallel section of 
$\,\mathcal{E}^{\odot2}\nh$, which may also be treated as a parallel
vec\-tor-bun\-dle en\-do\-mor\-phism of $\,\mathcal{E}\nh$, due to the
presence in $\,\mathcal{E}\,$ of the parallel fibre metric
induced by $\,\g$. This last en\-do\-mor\-phism acts on local parallel
sections
of $\,\mathcal{E}$ and, when the space of these 
sections is identified with $\,\mv\hs$ (see above), it becomes the
en\-do\-mor\-phism $\,A:\mv\nh\to\mv\hs$ of (\ref{dta}). (One sees this 
comparing formula (6) in \cite{derdzinski-roter-08} with the expression, 
in \cite[p.\ 93]{roter}, for the only pos\-si\-bly-non\-ze\-ro essential
component $\,W\hskip-2.5pt_{1i1\nh j}\w$ of the Weyl tensor $\,W\nnh$.)
We will therefore use the symbols $\,\mv$ and $\,A:\mv\nh\to\mv\hs$ for 
the space of parallel sections of
$\,\mathcal{E}\,$ over $\,\hm\nh$, and 
for the en\-do\-mor\-phism just described. Note that, due to simple
connectivity of $\,\hm\,$ and (\ref{tcf}), the bundle
$\,\mathcal{E}\,$ is trivialized by global
parallel sections.

By (\ref{tqp}) and (\ref{gpw}), any
$\,\gamma\in\mathrm{Iso}\nh^+\nh(\hm\nh,\g)\,$ pulls 
$\,w$, or $\,\xi\,$ appearing in (\ref{dua}), back to $\,qw\,$ or,
respectively, $\,q^{-\nnh1}\xi$, for
some $\,q\in(0,\infty)$. Since $\,A:\mv\nh\nnh\to\mv\hs$ as
interpreted above was the image of $\,\xi\hn\otimes\hs\xi\,$ under a natural
vec\-tor-bun\-dle morphism, the
isometry in question pulls $\,A\,$ back to $\,q^{-\nh2}\nh\nnh A$, while it
also acts as a linear isometry $\,B\,$ of the space $\,\mv\hs$ of parallel
sections. In terms of push-for\-wards rather than pull\-backs,
\begin{equation}\label{pfw}
B\nh AB^{-\nnh1}\,=\,\hs q^2\nh\nnh A\hh.
\end{equation}
\begin{lemma}\label{babqa}Let\/ $\,A\,$ be a nonzero linear
en\-do\-mor\-phism of a pseu\-do\hs-Euclid\-e\-an vector space\/ $\,\mv\nnh$.
If, for some\/ $\,q\in(0,\infty)\smallsetminus\{1\}\,$ there exists a linear
isometry\/ $\,B\,$ of\/ $\,\mv\hs$ satisfying\/ {\rm(\ref{pfw})} then, with
our fixed\/ $\,A$, such\/ $\,B\,$ exists for every\/ $\,q\in(0,\infty)$.
\end{lemma}
\begin{proof}For the space $\,\mathrm{End}\,\mv\hs$ of all linear 
en\-do\-mor\-phisms $\,\mv\nh\nnh\to\mv\nh$, the formula 
$\,\mathcal{J}=\{(q,B)\in\bbR\times\mathrm{End}\,\mv:(BB^*\nh,B\nh AB^*) 
=(\mathrm{Id},q^2\nh\nnh A)\}\,$ defines an algebraic variety
$\,\mathcal{J}\subseteq\bbR\times\mathrm{End}\,\mv\nh$.
By Whitney's classical result 
\cite[Theorem 3]{whitney}, $\,\mathcal{J}\,$ has finitely many connected 
components, and hence so does the intersection $\,K=K'\nh\cap(0,\infty)\,$ for
the image $\,K'$ of $\,\mathcal{J}\,$ under the projection
$\,(q,B)\mapsto q$. As $\,q^r\nh\in K\,$ whenever $\,\q\in K$ and 
$\,r\in\bbZ$, Remark~\ref{ccomp} now gives $\,K=(0,\infty)$.
\end{proof}
We have the following immediate consequence of (\ref{rns}):
\begin{equation}\label{rnz}
\mathrm{local\ homogeneity\ of\ }\,(M\nh,\g)\,\mathrm{\ or\
}\,(\hm\nh,\hg)\,\mathrm{\ implies\ that\ }\,\mathrm{Ric}\ne0\,\mathrm{\ 
everywhere.}
\end{equation}
If $\,\mathrm{Ric}\ne0\,$ everywhere, it follows from (\ref{ric}) --
(\ref{fcg}) that
\begin{equation}\label{ivf}
|f|^{1/2}dt\,\mathrm{\ is\ a\ closed\ }\,\Gm\nh\hyp\mathrm{in\-var\-i\-ant\
}\,C^\infty\nnh\ \,1\hyp\mathrm{form\ without\ zeros\ on\ }\,\hm.
\end{equation}
Closedness is here due to (\ref{flt}). Thus, in view of (\ref{wnt}), the
$\,C^\infty\nnh$ vector field $\,w'\nh=|f|^{1/2}\hs w\,$ is
$\,\pi$-pro\-ject\-a\-ble onto a vector field without zeros on 
$\,M=\hm\nnh/\hh\Gm\nh$, also denoted by $\,w'$ and, from
(\ref{wnt}), (\ref{ker}) and (\ref{flt}),
\begin{equation}\label{par}
\mathrm{on\ }\,M\nh\mathrm{,\ if\ }\,\mathrm{Ric}\ne0\,\mathrm{\ everywhere,\
}\,w'\hs\mathrm{\ spans\ }\,\,\mathcal{D}\hs\,\mathrm{\ and\ is\ parallel\
along\ }\,\hs\mathcal{D}^\perp\nnh.
\end{equation}
\begin{lemma}\label{gnobs}Let\/ $\,(\hm\nh,\hg)\,$ be the
pseu\-do\hs-Riem\-ann\-i\-an universal covering space of a compact rank-one
ECS manifold. If\/ $\,\hm\hs$ admits a closed\/ 
$\,\Gm\nh$-in\-var\-i\-ant\/ $\,C^\infty\nnh$ $\,1$-form without zeros which 
is a functional multiple of\/ $\,dt$, for the function\/ 
$\,t:\hm\to\bbR$ introduced in Section\/~{\rm\ref{an}}, then the leaves
of\/ $\,\mathcal{D}^\perp\hskip-2pt$ in\/ $\,\hm\hs$ are the factor manifolds
of a global product decomposition of\/ $\,\hm\nh$, and coincide with the
levels of\/ $\,t$.
\end{lemma}
\begin{proof}Up to a change of sign, the $\,1$-form in question may be
expressed as $\,\psi\,dt$, where $\,\psi:\hm\to(0,\infty)\,$ is locally a 
function of $\,t$. Choosing $\,\chi:\hm\to\bbR\,$ with 
$\,d\chi=\psi\,dt$, and denoting by $\,D\chi\,$ the $\,h$-grad\-i\-ent of 
$\,\chi$, where $\,h\,$ is the $\,\pi$-pull\-back of a Riemannian metric on 
$\,M\nh$, we see that $\,D\chi\,$ and $\,u=D\chi/[h(D\chi,D\chi)]\,$ are both
nonzero everywhere and complete (the latter since
$\,d\chi=h(D\chi,\,\cdot\,)$, with $\,d\chi\,$ and $\,h\,$ both
$\,\pi$-pro\-ject\-a\-ble onto $\,M$). By (\ref{ker}), the (possibly
disconnected) levels of $\,\chi$ are unions of leaves of
$\,\mathcal{D}^\perp\nnh$.

We now use Milnor’s argument \cite[p.\ 12]{milnor} involving the flow
$\,(\taw,x)\mapsto\phi\hs(\taw,x)$ of $\,u$. Remark~\ref{flows} may be
applied to a fixed (possibly disconnected) level $\,L\,$ of $\,\chi$ and the
restriction of the flow $\,\phi\,$ to $\,\bbR\times L$. The restriction
is not only locally dif\-feo\-mor\-phic, but bijective as well, since
$\,d_u\chi=1\,$ and so the parameter $\,\taw\,$ of each integral curve differs
from $\,\chi\,$ by a constant. Since $\,\bbR\times L$, dif\-feo\-mor\-phic 
to $\,\hm\nh$, must be connected, the assertion follows, the levels
of $\,t\,$ being the same as those of $\,\chi\,$ (and thus equal to
the leaves of $\,\mathcal{D}^\perp$) due to positivity of $\,\psi=d\chi/dt$.
\end{proof}
\begin{theorem}\label{lchom}Let\/
$\,\hm\nh,\g,\Gm\nh,f\nh,t,
\,(\hskip2.3pt)\hskip-2.2pt\dot{\phantom o}\nh=\hs d/dt\,$ and\/ 
$\,M=\hm\nnh/\hh\Gm\hs$ be as in Section\/~{\rm\ref{an}}. If\/ $\,M\hs$ is
compact, the following three conditions are mutually equivalent.
\begin{enumerate}
  \def\theenumi{{\rm\roman{enumi}}}
\item $(\hm\nh,\hg)$, or\/ $\,(M\nh,\g)$, is locally homogeneous.
\item $\fh\ne0\,$ everywhere and\/ 
$\,(|\fh|^{-\nnh1/2})\hskip-1.2pt\ddot{\phantom o}\nh=\hs0$.
\item $(|\fh|^{-\nnh1/2})\hskip-1.2pt\ddot{\phantom o}\nh=\hs0\,$ wherever\/ 
$\,\fh\ne0$.
\end{enumerate}
\end{theorem}
\begin{proof}First, (i) yields $\,\fh\ne0\,$ everywhere due to (\ref{rnz}) and 
(\ref{ric}). By Remark~\ref{relax}, $\,\g\,$ has, locally, the form
(\ref{pnz}) on some coordinate domain $\,\,U\nh$, and then, as shown in 
\cite[formula (3.3)]{derdzinski-terek-ne}, the existence of a Kil\-ling field 
$\,v\,$ on $\,\,U\hs$ with an every\-where-non\-zero component in the $\,t\,$ 
coordinate direction (which follows from local homogeneity) gives 
$\,(|\fh|^{-\nnh1/2})\hskip-1.2pt\ddot{\phantom o}\nh=0\,$ on $\,\,U\nh$,
that is, (ii), while (ii) trivially leads to (iii). Assuming (iii), we now
obtain (ii): $\,\fh\ne0\,$ everywhere, since otherwise we could choose local
coordinates as above around a boundary point of the zero set of $\,f\nh$, with
$\,t\,$ ranging over an open interval $\,I\hn'\nh$. (By (\ref{ric}) and
(\ref{rns}), $\,\fh\ne0\,$ somewhere.) A maximal open sub\-in\-ter\-val 
$\,I\hn'\nh'$ of
$\,I\hn'$ with $\,|\fh|>0\,$ on $\,I\hn'\nh'$ must equal $\,I\hn'\nh$, or else 
$\,I\hn'\nh'$ would have a finite endpoint lying in $\,I\hn'\nh$, at which 
$\,|\fh|^{-\nnh1/2}\nh$, being a linear function, would have a finite limit,
contrary to vanishing of $\,f\,$ at the endpoint due to maximality of 
$\,I\hn'\nh'\nh$. This contradiction shows that the zero set of $\,f\,$ has no
boundary points, proving (ii).

Suppose now that (ii) holds. According to (\ref{ivf}), the $\,1$-form 
$\hs|f|^{1/2}dt\hs$ satisfies the assumption, and hence the assertion, of
Lemma~\ref{gnobs}, so that the levels of $\,t\,$ are connected. We may thus
drop the word `locally' in 
(\ref{flt}) and, with $\,I\subseteq\bbR\,$ denoting the range of $\,t$,
view $\,f\,$ as a function $\,f:I\to\bbR\,$ which, since 
$\,(|\fh|^{-\nnh1/2})\hskip-1.2pt\ddot{\phantom o}\nh=0$ and
$\,f\,$ is nonconstant (Remark~\ref{relax}), has the form
$\,f(t)=\ve\hh(t-b)^{-\nh2}$ for some real constants $\,\ve\ne0\,$ and $\,b$. 
Thus, $\,I\nh\subseteq(-\infty,b)\,$ or $\,I\nh\subseteq(b,\infty)$. Now
(\ref{tqp}) leads to a homo\-mor\-phism
$\,\Gm\ni\gamma\mapsto(q,p)\in\mathrm{Af{}f}^+(\bbR)\,$ 
into the group of increasing af\-fine transformations of $\,\bbR\,$ mapping
$\,I\hs$ onto itself. As $\,I\nh\subseteq(-\infty,b)\,$ or
$\,I\nh\subseteq(b,\infty)$, the image of this homo\-mor\-phism cannot contain
a nontrivial translation, and must be infinite (or else it would be trivial,
causing $\,t\,$ to descend to a function without critical points on the
compact manifold $\,M=\hm\nnh/\hh\Gm$). Consequently, some $\,\gamma\in\Gm$
has $\,q\in(0,\infty)\smallsetminus\{1\}\,$ in (\ref{tqp}). Then  
(\ref{pfw}) and Lemma~\ref{babqa} imply that
\begin{equation}\label{inf}
\begin{array}{l}
\mathrm{for\ every\ }\,q\in(0,\infty)\smallsetminus\{1\}\,\mathrm{\ there\
 exists\ a\ linear}\\
\lr\hyp\mathrm{isometry\ }\,B:\mv\nh\nnh\to\mv\,\mathrm{\ with\
}\,B\nh AB^{-\nnh1}\hs=\,q^2\nh\nnh A\hh,
\end{array}
\end{equation}
where $\,A,\mv\hs$ and $\,\lr\,$ are a part of the data (\ref{dta})
representing the metric $\,\g\,$ in suitable coordinates around any point
of $\,\hm\nh$, (See the lines preceding (\ref{pfw}).) Also, instead of
$\,f(t)=\ve\hh(t-b)^{-\nh2}\nh$, we may require
that $\,f\,$ have the form
\begin{equation}\label{htf}
f(t)=\ve\hh t^{-\nh2}\,\mathrm{\ for\ all\ }\,t\in I\nh=(0,\infty)\hh,
\end{equation}
as we are free to modify our choice of $\,t\,$ via the 
af\-fine substitution replacing $\,t\,$ by $\,|t-b\hh|$. (The equality
$\,I\nh=(0,\infty)$, rather than just the inclusion $\,I\subseteq(0,\infty)$,
follows since an infinite group of increasing af\-fine transformations maps
$\,I\hs$ onto itself.)

Local homogeneity of $\,\g\,$ now follows: by (\ref{htf}), the
lo\-cal-co\-or\-di\-nate expression (\ref{pnz}) of $\,\g\,$ amounts to that
of the metric $\,g^P$ in \cite[top of p.\ 170]{derdzinski-78}, with our
coordinate $\,x^1\nh=t\,$ denoted there by $\,u^1\nh$. Our (\ref{inf}) now
becomes formula (10) in \cite[p.\ 172]{derdzinski-78} which, as stated there,
guarantees homogeneity of the metric $\,g^P$ on
$\,I\nh\times\bbR\times\mv\nh$, for $\,I\nh=(0,\infty)\,$ and
$\,\mv\nh=\bbR^{\hskip-.6ptn-2}\nnh$.

We have thus shown that (ii) implies (i), completing the proof.
\end{proof}
\begin{remark}\label{lorlh}A Lo\-rentz\-i\-an ECS manifold must have rank one
(see the Introduction) and is never 
locally homogeneous: (\ref{pfw}) with $\,q\ne\pm1\,$ implies nil\-po\-ten\-cy
of $\,A$, while in the Lo\-rentz\-i\-an case the fibre metric 
in $\,\mathcal{E}\hs$ induced by
$\,\g$, which corresponds to $\,\lr\,$ in (\ref{dta}), is positive definite,
and so the nonzero $\,\lr$-self-ad\-joint\ linear\ en\-do\-mor\-phism
$\,A:\mv\nh\nnh\to\mv\nh$, being di\-ag\-o\-nal\-iz\-able, cannot be
nil\-po\-tent.
\end{remark}

\section{Function spaces and the first real co\-ho\-mol\-o\-gy}\label{fs}
\setcounter{equation}{0}
We refer to a continuous $\,1$-form $\,\zeta$ on a manifold $\,M\hs$ as {\it 
closed\/} if it is locally exact in the sense that every point of $\,M\,$ has
a neighborhood $\,\,U\hh$ with $\,\zeta=d\psi\,$ on $\,\,U\hh$ for some
$\,C^1$ function $\,\psi:U\to\bbR$. 

Due to the universal coefficient theorem \cite[p.\ 378, Theorem 13.43]{lee},
for any manifold $\,M\,$ one has an isomorphic identification 
$\,H^1\nh(M\nh,\bbR)=\mathrm{Hom}\hh(\pi\nh_1\w\hn M\nh,\hs\bbR)$. 

As in the case of smooth $\,1$-forms, a closed continuous $\,1$-form
$\,\zeta\,$ on $\,M\hs$ thus gives rise to the
co\-ho\-mol\-o\-gy class $\,\lsq\hh\zeta\rsq\in H^1\nh(M\nh,\hs\bbR)\,$ which,
as a homo\-mor\-phism $\,\pi\nh_1\w\hn M\to\bbR$, assigns to each homotopy
class of piecewise $\,C^1$ loops at a fixed base point the integral of
$\,\zeta$ over a representative loop. Clearly, $\,\lsq\hh\zeta\rsq=0\,$ if and
only if $\,\zeta=d\psi\,$ for some $\,C^1$ function $\,\psi:M\to\bbR$.

The following lemma uses the notations $\,\pi,\hm\nh,t,f\nh,\Gm\hs$ introduced
in Section~\ref{an}.
\begin{lemma}\label{rvsof}Let\/ $\,\mathcal{F}\,$ be the real vector space
formed by all continuous functions\/ $\,\chi:\hm\to\bbR\,$ such that the\/
$\,1$-form\/ $\,\chi\nnh\,dt\,$ is closed and
\begin{equation}\label{feq}
\chi\circ\gamma=q^{-\nnh1}\chi\,\mathrm{\ whenever\ 
}\,\gamma\in\Gamma\,\mathrm{\ and \ }\,q\in(0,\infty)\,\mathrm{\ is\ the\
}\,q\hyp\mathrm{im\-a\-ge\ of\ }\,\gamma\hh.
\end{equation}
Then\/
$\,|f|^{1/2}\in\mathcal{F}\,$ and\/ $\,|\dot f|^{1/3}\in\mathcal{F}\nh$.
Furthermore, if\/ $\,\chi\in\mathcal{F}\nh$, then the\/
$\,1$-form\/ $\,\chi\nnh\,dt$ is $\,\Gm\nh$-in\-var\-i\-ant, and hence\/ 
$\,\pi$-pro\-ject\-a\-ble onto a closed\/ $\,1$-form
on\/ $\,M=\hm\nnh/\hh\Gm\nh$, also denoted by\/ $\,\chi\nnh\,dt$, which gives
rise to a linear operator
\begin{equation}\label{lop}
\mathcal{F}\ni\chi\,\mapsto\,\lsq\chi\nnh\,dt\rsq\in H^1\nh(M\nh,\hs\bbR)\hh.
\end{equation}
\end{lemma}
This is a trivial consequence of (\ref{fcg}).

\section{Existence of special functions}\label{es}
\setcounter{equation}{0}
\begin{theorem}\label{admts}Let\/ $\,(M\nh,\g)\,$ be a compact rank-one 
ECS manifold such that\/ $\,\mathcal{D}^\perp\hskip-2pt$ is 
trans\-ver\-sal\-ly o\-ri\-ent\-able. If\/ $\,(M\nh,\g)$ is not locally
homogeneous, then there exists a nonconstant\/ $\,C^1\nnh$ function\/  
$\,\mu:M\to\bbR$, constant along\/ $\,\mathcal{D}^\perp\nnh$.
\end{theorem}
\begin{proof}We will show that,
for $\,\mathcal{F}\,$ defined in Lemma~\ref{rvsof}, either 
$\,\dim\mathcal{F}\nh<\infty$ and $\,(M\nh,\g)\,$ is locally 
homogeneous, or $\,\dim\mathcal{F}\nh=\infty\,$ and such $\,\mu\,$ exists.

First, let $\,\dim\mathcal{F}\nh<\infty$. In this case, even without using
compactness of $\,M\nh$, we can apply Lemma~\ref{basis} to $\,X=M\nh$, our 
vector space $\,\mathcal{F}\nh$, and the $\,m$-ar\-gu\-ment operation 
$\,\varPi\,$ sending $\,\psi_1\w,\dots,\psi_m\w$ to the
product of powers of the absolute values
$\,|\psi_1\w|,\dots,|\psi_m\w|\,$ with any fixed
positive exponents adding up to $\,1\,$ (for instance, the geometric mean of
the absolute values). As $\,|f|^{1/2},|\dot f|^{1/3}$ lie in
$\,\mathcal{F}\,$ (see Lemma~\ref{rvsof}), on each of the open sets
$\,X\hskip-2pt_j\w\subseteq M\,$ obtained in Lemma~\ref{basis},
$\,|f|^{1/2}$ and $\,|\dot f|^{1/3}$ are constant multiples of the $\,j\hh$th
function $\,\chi\nnh_j\w$ from the basis
$\,\chi_1\w,\dots,\chi_m\w$ of $\,\mathcal{F}\nh$, which
clearly gives $\,(|\fh|^{-\nnh1/2})\hskip-1.2pt\ddot{\phantom o}\nh=0\,$
wherever $\,f\ne0$. (Note that $\,|f|^{1/2}\nh$, being a linear combination
of the functions $\,\chi_1\w,\dots,\chi_m\w\nh$,
vanishes on their simultaneous zero set
$\,X\nh\smallsetminus\bigcup_{j=1}^mX\hskip-2pt_j\w$.) Local homogeneity is
now immediate from Theorem~\ref{lchom}.

Finally, suppose that $\,\dim\mathcal{F}\nh=\infty$. Due to compactness of
$\,M\nh$, (\ref{lop}) is noninjective and we may fix
$\,\chi\in\mathcal{F}\smallsetminus\{0\}\,$ lying in its kernel, so that
$\,\chi\nnh\,dt\,$ treated as a $\,1$-form on $\,M\,$ equals 
$\,d\mu\,$ for some (nonconstant) $\,C^1$ function $\,\mu:M\to\bbR$.
By (\ref{ker}), this completes the proof.
\end{proof}
\begin{remark}\label{sards}
Sard's theorem \cite[Theorem 1.3 on p.\ 69]{hirsch} normally applies to 
$\,C^k$ mappings from an $\,n$-man\-i\-fold into an $\,m$-man\-i\-fold, with
$\,k\ge\mathrm{max}\hs(n-m+1,1)$, guaranteeing that the critical values form a
set of zero measure. In our case, even though $\,\mu:M\to\bbR$ is only of
class $\,C^1\nnh$, and $\,M\,$ can have any dimension $\,n\ge4$, the same
conclusion remains valid, and so, due to compactness of $\,M\nh$,
\begin{equation}\label{rng}
\mu(M)\,\mathrm{\ contains\ an\ open\ interval\ consisting\ of\ regular\
values\ of\ }\,\mu\hh,
\end{equation}
$\mu(M)\,$ being being the range of $\,\mu$. In fact,
(\ref{wnt}) gives $\,dt\ne0\,$ everywhere in $\,\hm\nh$, and so 
$\,M\,$ is covered by
finitely many connected open sets $\,\,U\,$ each of which can be
dif\-feo\-mor\-phic\-al\-ly identified with an open set
$\,\,\widehat{U}\subseteq\hm\,$ such that the levels of 
$\,t:\widehat{U}\to\bbR\,$ are all connected. This turns $\,\mu$ restricted 
to $\,\,U\,$ into a function of $\,t$, allowing us to use Sard's
theorem as stated above for $\,k=n=m=1$.
\end{remark}

\section{Holonomy of compact leaves}\label{hc} 
\setcounter{equation}{0}
Let $\,(M\nh,\g)\,$ be a rank-one ECS manifold. Its rank-one 
Ol\-szak distribution $\,\mathcal{D}$ (see the Introduction), being
parallel, carries a linear connection induced from the Le\-vi-Ci\-vi\-ta
connection of $\,\g$, and this connection is flat since, locally, 
$\,\mathcal{D}\,$ is spanned by the parallel gradient $\,\nabla\hn t$, cf.\
(\ref{wnt}).

For any leaf $\,L\,$ of $\,\mathcal{D}^\perp\nnh$, this flat connection 
gives rise to one in the line bundle
$\,\mathcal{D}\hskip-2pt_L\w$ arising as the restriction of $\,\mathcal{D}\,$
to $\,L\,$ and, consequently, also to what we call
\begin{equation}\label{tfl}
\mathrm{the\ flat\ connection\ in\ the\ line\ bundle\ }\,\dla\mathrm{\ dual\
to\ }\,\mathcal{D}\hskip-2pt_L\w\hh.
\end{equation}
Note that $\,\dla$ is canonically isomorphic to the normal
bundle of $\,L\,$ in $\,M\nh$, since $\,\mathcal{D}$ 
is isomorphic to the dual of $\,T\nh M/\mathcal{D}^\perp\nnh$, via the
isomorphism assigning to $\,v\in\mathcal{D}\nnh_x\w$, at any 
$\,x\in M\nh$, the linear functional $\,\g_x\w(v,\,\cdot\,):\txm\nh\to\bbR$, 
vanishing on $\,\mathcal{D}\nnh_x^\perp\nnh$.

The next result remains valid without transversal orientability of
$\,\mathcal{D}^\perp\nnh$ or compactness of $\,M\nh$. We make these
assumptions here just to simplify the discussion.
\begin{theorem}\label{hlncl}Let\/ $\,L\,$ be a compact leaf of\/
$\,\mathcal{D}^\perp\hskip-2pt$ in a compact rank-one ECS manifold\/
$\,(M\nh,\g)$. If\/ $\,\mathcal{D}^\perp\hskip-2pt$ is trans\-ver\-sal\-ly
o\-ri\-ent\-able, then some neighborhood\/ $\,\,U$ of\/ $\,L\,$ in\/
$\,M\,$ can be dif\-feo\-mor\-phic\-al\-ly identified with a neighborhood\/
$\,\,U'$ of the zero section\/ $\,L\,$ in the line bundle\/
$\,\dla$ so as to make the distribution\/
$\hs\mathcal{D}^\perp\hskip-2pt$ on\/ $\,\,U\nh$ correspond to the restriction
to\/ $\,\,U'$ of the horizontal distribution of the flat connection in\/
$\,\dla$.
\end{theorem}
\begin{proof}Let $\,\,U=\phi\hs((-\ve,\ve)\times L)\,$ for the flow $\,\phi\,$
of a fixed $\,C^\infty\nnh$ vector field $\hs v\hs$ on $\,M\hs$ which is
nowhere
tangent to $\,\mathcal{D}^\perp\nnh$, and for $\,\ve\,$ chosen as in
Remark~\ref{flows}. We define the required 
dif\-feo\-mor\-phism $\,\varPsi:U\to U'$ by declaring how  
$\,\varPsi(\phi\hs(\taw,x))\,$ depends on $\,(\taw,x)\in(-\ve,\ve)\times L$, 
cf.\ Remark~\ref{flows}. For any point
$\,y\in\pi^{-\nnh1}(x)$, with $\,\pi\,$ as in (\ref{ucp}), the flow
$\,\hat\phi\,$ of the vector field $\,\hat v\,$ on $\,\hm\,$ projecting 
under $\,\pi\,$ onto $\,v$, and the parallel section $\,\xi\,$ of the line
bundle $\,\mathcal{D}^*$ over $\,\hm\,$ appearing in (\ref{dua}), we set
\begin{equation}\label{ptx}
\varPsi(\phi\hs(\taw,x))\,
=\,[\hh t(\hat\phi\hs(\taw,y))-t(y)]\,\xi_y\w\nh\circ(d\pi\nh_y\w)^{-\nnh1}\,
\in\,\mathcal{D}_{\nnh x}^*\,\subseteq\,\dla.
\end{equation}
Here $\,\xi_y\w\in\mathcal{D}_{\nnh y}^*$ is a linear functional on 
$\,\mathcal{D}_{\nnh y}\w\subseteq\tyhm\,$ and we compose it with the inverse
of the isomorphism $\,d\pi\nh_y\w:\tyhm\to\txm\nh$, so that the result is 
a functional on $\,\mathcal{D}_{\nnh x}\w\subseteq\txm\nh$, which we then 
multiply by the scalar $\,t(\hat\phi\hs(\taw,y))-t(y)$. The fibres
$\,\mathcal{D}_{\nnh x}$ of the line bundle $\,\dla$ over $\,L\,$ are
treated here as pairwise disjoint subsets of the total space, also denoted by 
$\,\dla$. 

First, (\ref{ptx}) does not depend on the choice of $\,y\in\pi^{-\nnh1}(x)$.
In fact, replacing $\,y$ by $\,\gamma(y)$, with $\,\gamma\in\Gm\nh$, cf.\ 
(\ref{ucp}), we get the same right-hand side in (\ref{ptx}), since
\[
t(\hat\phi\hs(\taw,\gamma(y)))\nh-\nh t(\gamma(y))\nh
=\nh q\hs[\hh t(\hat\phi\hs(\taw,y))\nh-\nh t(y)],\hskip7pt
\xi_{\gamma(y)}\w\nh\circ(d\pi\nh_{\gamma(y)}\w)^{-\nnh1}\nh
=\nh q^{-\nnh1}\xi_y\w\nh\circ(d\pi\nh_y\w)^{-\nnh1}
\]
for some real $\,q>0$. Namely, as $\,\gamma\,$ leaves $\,\hat v\,$
invariant, $\,\hat\phi\hs(\taw,\gamma(y))=\gamma(\hat\phi\hs(\taw,y))$, and so 
the relation $\,t\circ\gamma=qt+p\,$ in (\ref{tqp}) yields the first equality
displayed above. At the same time (\ref{gpw}) and (\ref{dua}) give
$\,q^{-\nnh1}\xi=\gamma^*\xi\,$ which, since $\,\pi=\pi\circ\gamma$, leads 
to $\,q^{-\nnh1}\xi_y\w=\xi_{\gamma(y)}\w\nnh\circ d\gamma\nh_y$ and 
$\,d\pi\nh_y\w=d\pi\nh_{\gamma(y)}\w\nh\circ\hh d\gamma\nh_y\w$, so that 
$\,q^{-\nnh1}\xi_y\w\nh\circ(d\pi\nh_y\w)^{-\nnh1}
=\xi_{\gamma(y)}\w\nh\circ(d\pi\nh_{\gamma(y)}\w)^{-\nnh1}\nh$.

Smoothness of $\,\varPsi\,$ is obvious if one uses $\,y\,$ depending on
$\,x\,$ via a local inverse of $\,\pi$. 
Also, $\,\varPsi\,$ is a
fibre-pre\-serv\-ing mapping from $\,\,U=\phi\hs((-\ve,\ve)\times L)\,$
(viewed, when identified with $\,(-\ve,\ve)\times L$, as a trivial bundle with 
one-di\-men\-sion\-al fibres) into the line bundle $\,\dla$, operating as the
identity on the base manifold $\,L$, and constituting an embedding of each
fibre separately: by (\ref{ker}),
$\,|\hs d\hs[\hh t(\hat\phi\hs(\taw,y))]/dt|>0$. 
This makes $\,\varPsi\,$ itself an embedding.

Finally, with $\,y\,$ near $\,y\nh_*\w$ smoothly depending on $\,x\in L\,$
near some fixed $\,x_*\w$ as before, 
$\,t(y)=t(y\nh_*\w)$ is constant, by (\ref{ker}). Requiring that an
assignment $\,y\mapsto\taw(y)$ give $\,t(\hat\phi\hs(\taw(y),y))=t_*\w$
for a constant $\,t_*\w$ near $\,t(y\nh_*\w)$, and so 
$\,y\mapsto\lambda(y)=\hat\phi\hs(\taw(y),y)$ sends a neighborhood of
$\,y\nh_*\w$ in the leaf $\,L\nh_y\w$ into a single leaf, yields
(Remark~\ref{flmap}) a dif\-feo\-mor\-phism $\,\lambda\,$ between neighborhoods
of $\,y\nh_*\w$ and $\,\lambda(y\nh_*\w)\,$ in the two leaves. By 
(\ref{ptx}), $\,\varPsi(\pi(\lambda(y)))
=[\hh t_*\w\nnh-t(y\nh_*\w)]\,\xi_y\w\nh\circ(d\pi\nh_y\w)^{-\nnh1}$ which,
due to constancy of $\,t_*\w\nnh-t(y\nh_*\w)$, is a parallel local section of
$\,\dla$. This completes the proof.
\end{proof}
The holonomy representation of (\ref{tfl}) assigns to each $\,x\in L$ and each
homotopy class of piecewise $\,C^1$ loops at $\,x\,$ in $\,L\,$ a linear
automorphism of the line $\,\mathcal{D}\nnh_x\w$, that is, the multiplication
by some $\,q\in\bbR\smallsetminus\{0\}$. Since this is a multiple of the
identity, the {\it holonomy group\/} $\,H\hskip-2.5pt_L\w$ of the flat
connection (\ref{tfl}) in the line bundle $\,\dla$ over $\,L$, formed by all
these $\,q$, does not depend on $\,x$. Obviously,
\begin{equation}\label{hlg}
\mathrm{the\ holonomy\ group\ }\,H\hskip-2.5pt_L\w\mathrm{\ is\ either\
infinite,\ or\ trivial.}
\end{equation}
\begin{theorem}\label{eithr}In any rank-one ECS manifold\/ $\,(M\nh,\g)\,$ such
that\/ $\,\mathcal{D}^\perp\hskip-2pt$ is 
trans\-ver\-sal\-ly o\-ri\-ent\-able,  
condition\/ {\rm(\ref{cpl})} holds for\/
$\,\mathcal{V}=\mathcal{D}^\perp\nnh$. In addition, the two possibilities 
named in\/ {\rm(\ref{cpl})} correspond precisely to the two cases of\/
{\rm(\ref{hlg})}.
\end{theorem}
\begin{proof}Theorem~\ref{hlncl} allows us to treat some neighborhood
$\,\,U'$ of the zero section $\,L\,$ in the line bundle 
$\,\dla$ as a neighborhood of $\,L\,$ in $\,M\nh$.

Any fixed fibre metric in the line bundle $\,\dla$ gives
rise to the norm function $\,N:\dla\to[\hs0,\infty)\,$ on
the total space $\,\dla$, and to radius $\,\ve\,$
interval sub\-bundles 
$\,\,U\hskip-2.6pt_\ve\w=N^{-\nnh1}\nnh([\hs0,\ve))
\subseteq\dla$, where
$\,\ve>0$. As $\,L\,$ is compact, $\,\,U\hskip-2.6pt_\ve\w\subseteq U'$ for
$\,\ve\,$ near $\,0$, which turns $\,\,U\hskip-2.6pt_\ve\w$ into a
neighborhood of $\,L\,$ in $\,M\nh$, and is expressed as
$\,\,U\hskip-2.6pt_\ve\w\subseteq M\nh$.

First let $\,\,H\hskip-2.5pt_L\w$ be trivial. The total space of $\,\dla$ is
thus the union of global parallel sections obtained from one such section (on
which the norm function $\,N\hs$ has some maximum value $\,r>0$) via
multiplication by constants $\,a\ne0$. If $\,|\hs a\hh|\in(0,\ve/r)$, the
resulting parallel sections are thus contained in
$\,\,U\hskip-2.6pt_\ve\w\subseteq M\nh$, and hence constitute compact leaves
of $\,\mathcal{D}^\perp\nnh$. This is the second option in the ei\-ther-or
clause of (\ref{cpl}).

Next, let $\,\,H\hskip-2.5pt_L\w$ be infinite. We fix
$\,q\in \,H\hskip-2.5pt_L\w\nh\cap\hs(0,1)\,$ and, for any
$\,x\in L$, choose a piecewise $\,C^1$ loop $\,\lp_x\w$ at $\,x\,$ in
$\,L\,$ such that the parallel transport along $\,\lp_x\w$ in the line
bundle $\,\dla$ equals the multiplication by
$\,q\,$ in the line $\,\dxa$. For any
$\,u\in\dxa\hn\cap\hs U\hskip-2.6pt_\ve\w$ the horizontal lift
$\,\lf_x\w$ of $\,\lp_x\w$ with the initial point $\,u\,$ has the terminal
point $\,qu$. Treating $\,\lf_x\w$ as a compact subset of the total space
$\,\dla$, on which the norm function $\,N\hs$ has some
maximum value $\,r\nnh_x\w>0$, we may form the union 
$\,Q\nh_x\w\nh\nnh=\bigcup\nh_{i=k}^\infty\hs q^i\nh\lf_x\w$ for the least
integer $\,k\ge0\,$ with $\,q\hh^kr\nnh_x\w<\ve$. Thus,
$\,Q\nh_x\w\subseteq\hs U\hskip-2.6pt_\ve\w$ is a union of piecewise $\,C^1$
horizontal curves in the total space $\,\dla$, joined
end-to-end, and the same is true for $\,aQ\nh_x\w$ whenever
$\,a\in(-\nnh1,1)\smallsetminus\{0\}\,$ which, as
$\,\,U\hskip-2.6pt_\ve\w\subseteq M\nh$, makes 
$\,aQ\nh_x\w$ a subset of a leaf $\,L_{(a)}\w$ of
$\,\mathcal{D}^\perp\hskip-.5pt$ with
$\,L_{(a)}\w\nnh\subseteq\hs U\hskip-2.6pt_\ve\w\smallsetminus L$. Each
$\,L_{(a)}\w$ contains
the sequence $\,aq^i\nh u$, with integers $\,i\ge k$,
and $\,aq^i\nh u\,$ converges as $\,i\to\infty\,$ to $\,x\in L\,$ (the zero
vector in the line $\,\dxa$), so that the leaves
$\,L_{(a)}\w$ are all noncompact. At the same time, $\,aq\hh^k\nh u\in L_{(a)}\w$, with
our fixed $\,k\,$ and all $\,a\in(-\nnh1,1)\smallsetminus\{0\}$. Such
$\,aq\hh^k\nh u\,$ form a neighborhood of $\,0$, with $\,0$ itself removed,
in the line $\,\dxa$. Thus, in the portion of the radius
$\,\ve\,$ interval bundle $\,\,U\hskip-2.6pt_\ve\w$ over a neighborhood of
$\,x\,$ in $\,L\,$ on which $\,\dla$ has a nonzero parallel section,
the products of this parallel section by all real numbers sufficiently close 
to $\,0\,$ fill a neighborhood $\,\,U\nh(x)\,$ of $\,x\,$ in $\,M\,$ and the
leaves of $\,\mathcal{D}^\perp$ intersecting
$\,\,U\nh(x)\nh\smallsetminus L\,$ 
are all noncompact. Compactness of $\,L\hs$ allows us to choose finitely many
$\,x\in L\,$ such that $\,L\,$ is contained in the union
$\,\,U\hs$ of the corresponding sets $\,\,U\nh(x)$, which yields the first 
option in the ei\-ther-or clause of (\ref{cpl}).
\end{proof}

\section{Proofs of Theorems~\ref{maith} and~\ref{univc}}\label{pt} 
\setcounter{equation}{0}
To establish Theorem~\ref{maith}, fix a compact rank-one ECS manifold 
$\,(M\nh,\g)$. Passing to a two-fold isometric covering of $\,(M\nh,\g)$, if
necessary, we may also assume trans\-ver\-sal o\-ri\-enta\-bil\-i\-ty
of $\,\mathcal{D}^\perp\nnh$. Theorem~\ref{eithr} now implies (\ref{cpl}) for
$\,\mathcal{V}=\mathcal{D}^\perp\nnh$.

In addition, under the hypotheses of Theorem~\ref{maith}, there 
exists a compact leaf $\,L\,$ of $\,\mathcal{D}^\perp$ realizing the
second possibility in (\ref{cpl}), the one where some prod\-uct-like 
neighborhood of $\,L\,$ in $\,M$ is a union of compact leaves of
$\,\mathcal{D}^\perp\nnh$.

To prove this, note that either $\,(M\nh,\g)\,$ is locally homogeneous and
$\,\mathcal{D}^\perp\nnh$ has a compact leaf, or $\,(M\nh,\g)\,$ is not locally
homogeneous.

In the latter case, Theorem~\ref{admts} allows us to choose
a nonconstant $\,C^1$ function $\,\mu:M\to\bbR$, constant along
$\,\mathcal{D}^\perp\nnh$, and Remark~\ref{sards} implies (\ref{rng}). Letting
$\,\mu(x)\,$ be a regular value of $\,\mu$, and fixing a 
one-di\-men\-sion\-al sub\-man\-i\-fold of $\,M\hs$ containing $\,x$ and
trans\-verse to $\,\mathcal{D}^\perp\nnh$, we see that for every point $\,y\,$
in this sub\-man\-i\-fold lying sufficiently close to $\,x$, the connected
component, containing $\,y$, of the $\,\mu$-pre\-im\-age of $\,\mu(y)\,$ is a 
compact leaf of $\,\mathcal{D}^\perp\nnh$. This causes $\,L\,$ to satisfy the
second option in (\ref{cpl}) by obviously precluding the first one.

Consider now the former case: local homogeneity along with the existence of a 
compact leaf $\,L$. 
By (\ref{rnz}) and (\ref{par}), the line bundle $\,\dla$ with the connection
(\ref{tfl}) has the global parallel section $\,w'$ and, consequently, its
holonomy group $\,H\hskip-2.5pt_L\w$ is trivial. This leads again, via
Theorem~\ref{eithr}, to the second option in (\ref{cpl}).

Theorem~\ref{maith} is now immediate from
Theorem~\ref{bdcir} applied to $\,\mathcal{V}=\mathcal{D}^\perp\nnh$.

As pointed out in the Introduction, Theorem~\ref{univc} trivially follows 
from Theorem~\ref{maith} except when $\,(\hm\nh,\hg)\,$ is locally 
homogeneous. On the other hand, in the lo\-cal\-ly-ho\-mo\-ge\-ne\-ous case,
(\ref{rnz}) and (\ref{ivf}) imply that $\,|f|^{1/2}dt\,$ is a closed
$\,\Gm\nh$-in\-var\-i\-ant $\,C^\infty\nnh$ $\,1$-form without zeros on
$\,\hm$. Theorem~\ref{univc} is now obvious from Lemma~\ref{gnobs}.

\section{Further consequences}\label{fc}
\setcounter{equation}{0}
Let $\,(\hm\nh,\hg)\,$ be the pseu\-do\hs-Riem\-ann\-i\-an universal covering
space of a compact rank-one ECS manifold $\,(M\nh,\g)$, with the Ol\-szak
distribution $\,\mathcal{D}$, and the universal covering projection
$\,\pi:\hm\to M=\hm\nnh/\hh\Gm$, cf.\ (\ref{ucp}). As in Section~\ref{an}, we
assume transversal orientability of the orthogonal complement
$\,\mathcal{D}^\perp\nnh$.

For future reference, we state here three consequences of the above
assumptions.

First, $\,\hm\,$ {\it admits a smooth positive function\/
$\,\psi:\hm\to(0,\infty)\,$ for which the\/ $\,1$-form $\,\psi\,dt\,$ is
both\/ $\,\Gm\nh$-in\-var\-i\-ant} 
(in other words, $\,\pi$-pro\-ject\-a\-ble onto $\,M$), {\it and closed}
(which amounts to requiring that $\,\psi\,$ be, locally, a function of $\,t$).

In fact, in the lo\-cal\-ly-ho\-mo\-ge\-ne\-ous case (or, more generally,
when $\,\mathrm{Ric}\ne0$ everywhere), (\ref{ivf}) allows us to choose
$\,\psi=|f|^{1/2}\nh$.

If $\,(M\nh,\g)\,$ is not locally homogeneous, $\,\mathcal{D}^\perp\nnh$, on
$\,M\nh$, is the vertical distribution of a fibration $\,M\nh\to S^1\nh$.
(This is Theorem~\ref{maith}.) We obtain our $\,\psi\,dt$, or its opposite,
by pulling
back from $\,S^1$ to $\,M\,$ a smooth $\,1$-form without zeros.

Secondly, {\it the parallel vector field\/ $\,w=\nabla\hn t\,$ on 
$\,\hm\nh$, spanning\/ $\,\mathcal{D}$, which appears in\/} (\ref{wnt}), {\it
is complete}.
Namely, for $\,\psi:\hm\to(0,\infty)\,$ as above, $\,\Gm\nh$-in\-var\-i\-ance
of $\,\psi\,dt$ implies the same for $\,\psi w\,$ (since $\,w=\nabla\hn t$,
that is, $\,dt=\g(w,\,\cdot\,)$). Completeness of $\,\psi w\,$ now follows
due to its resulting $\,\pi$-pro\-ject\-a\-bil\-i\-ty onto the compact manifold
$\,M\nh$. However, our $\,\psi\,$ is, locally, a function of $\,t\,$ and, by
(\ref{ker}), $\,\mathcal{D}^\perp\nnh=\hs\mathrm{Ker}\,dt\,$ on $\,\hm\nh$. 
This makes $\,\psi\,$ constant along every leaf of $\,\mathcal{D}^\perp$ and
$\,w\,$ tangent to the leaf. The integral curves of $\,\psi w\,$ are thus
af\-fine re\-pa\-ram\-e\-tri\-za\-tions of those of $\,w$, and so $\,w$ is
complete as well.

Finally, {\it the levels of\/ $\,t:\hm\nh\to\bbR\,$ are all connected, and
coincide with the leaves of\/ $\,\mathcal{D}^\perp\hskip-2pt$ in\/}
$\,\hm\nh$.  Thus, if $\,\chi:\hm\nh\to\bbR\,$ is locally a function of $\,t$,
it must also be one globally, in the sense of being a composite of $\,t\,$
with some function $\,t(\hm)\to\bbR$, which applies, in particular, to
$\,\chi=f\,$ appearing in (\ref{ric}) -- (\ref{flt}).

To see this, use Theorem~\ref{univc}: the leaves of
$\,\mathcal{D}^\perp\hskip-2pt$ in $\,\hm\,$ are the factor manifolds of a
global product decomposition of $\,\hm\nh$, some open interval 
$\,I\hn'\nh\subseteq\bbR\,$ being the one-di\-men\-sion\-al factor. The
leaves are thus connected, and $\,t:\hm\nh\to\bbR$, 
constant along them due to (\ref{ker}), and having a nonzero parallel
gradient -- see (\ref{wnt}) -- descends to a strictly monotone function 
$\,I\hn'\nh\to\bbR$, the levels of which thus are single points. This makes the
levels of $\,t\hh$ equal to single leaves of 
$\,\mathcal{D}^\perp\nnh$, and hence connected.

\section*{Appendix: The Lo\-rentz\-i\-an case}\label{lc}
\setcounter{equation}{0}
From now on we always assume transversal orientability of
$\,\mathcal{D}^\perp\nnh$, for the Ol\-szak distribution $\,\mathcal{D}\,$ of
the ECS manifold $\,(M\nh,\g)\,$ in question, which makes $\,\mathcal{D}\,$
a trivial line bundle over $\,M\nh$.

The purpose of this section is to inform the reader how the proof of 
Theorem~\ref{maith} differs from that for Lo\-rentz\-i\-an ECS manifolds 
in \cite{derdzinski-roter-08}, and specifically what issues arise in higher
signatures and how they are dealt with. We already explained in the 
Introduction why Theorem B of the paper \cite{derdzinski-roter-08} -- 
the Lo\-rentz\-i\-an case of our Theorem~\ref{maith} -- does not require 
assuming rank one or excluding local homogeneity.

Let us begin by outlining the proof of \cite[Theorem B]{derdzinski-roter-08},
given in \cite{derdzinski-roter-08}. First -- as
we pointed out in the lines following (\ref{ucp}), and in (\ref{tcf}) -- in 
any rank-one ECS manifold $\,(M\nh,\g)\,$ the Le\-vi-Ci\-vi\-ta connection
$\,\nabla\hs$ induces natural flat connections both in the Ol\-szak
distribution $\,\mathcal{D}\,$ and in the quotient bundle 
$\,\mathcal{E}=\mathcal{D}^\perp\hskip-2.3pt/\mathcal{D}\,$ 
over $\,M\nh$, while -- see \cite[Sect.\,4]{derdzinski-roter-08} -- the Weyl
tensor $\,W\hs$ leads to a vec\-tor-bun\-dle morphism 
$\,\varPhi:(\mathcal{D}^*)^{\otimes2}\to\hs(\mathcal{E}^*){}^{\otimes2}$ with
$\,\varPhi\hskip-2pt_x\w(\lambda\otimes\lambda'\hh):\mathcal{E}\nh_x\w
\times\mathcal{E}\nh_x\w\to\bbR$, for any $\,x\in M\,$ and
$\,\lambda,\lambda'\in\dxa$, equal to the 
symmetric bi\-lin\-e\-ar form sending the cosets $\,v+\mathcal{D}\nh_x\w$ and 
$\,v\hh'\nh+\mathcal{D}\nnh_x\w$ of vectors
$\,v,v\hh'\nh\in\mathcal{D}\nnh_{\nh x}^{\hs\perp}$
to $\,W\hskip-3pt_x\w(v,u,u\hh'\nnh,v\hh'\hs)$, where $\,u,u\hh'\nh\in\txm\,$
are any vectors with $\,\lambda=g_x\w(u,\,\cdot\,)\,$ and
$\,\lambda'\nh=g_x\w(u\hh'\nnh,\,\cdot\,)\,$ on $\,\mathcal{D}\nnh_x\w$. As 
observed in \cite[Sect.\,4]{derdzinski-roter-08}, the morphism
$\,\varPhi\,$ is well-defined, parallel relative to natural flat the
connections in the bundles involved, and nonzero (which makes it injective)
at every point $\,x\in M$. So far the metric signature of $\,\g\,$ was
arbitrary: but when it is Lo\-rentz\-i\-an, the fibre metric in
$\,\mathcal{E}=\mathcal{D}^\perp\hskip-2.3pt/\mathcal{D}\,$ induced by $\,\g\,$
is positive definite, leading, via injectivity of $\,\varPhi$, to a parallel
fibre norm $\,|\hskip2.3pt|\,$ in the line bundle $\,\mathcal{D}$. Cf.\
\cite[the end of Sect.\,4]{derdzinski-roter-08}. This proves 
\cite[Theorem D]{derdzinski-roter-08}: for the Ol\-szak
distribution $\,\mathcal{D}\,$ of any Lo\-rentz\-i\-an ECS 
manifold, it follows (from transversal orientability of
$\,\mathcal{D}^\perp$) that
\begin{enumerate}
\item[(a)] $\mathcal{D}\,$ is spanned  by a global parallel section $\,w$,
namely, one with $\,|w|=1$.
\end{enumerate} 
We can now paraphrase the remainder of the proof of
\cite[Theorem B]{derdzinski-roter-08}, which consists of the paragraph
preceding \cite[Remark 5.1]{derdzinski-roter-08}, followed by
\cite[Lemma 1.2]{derdzinski-roter-08}, and instead of assuming the
Lo\-rentz\-i\-an signature, uses only (a); the symbols
$\,\xi,u,\rho,\phi\,$ and $\,\psi\hskip-2pt=\hskip-2pt\theta$ of 
\cite{derdzinski-roter-08} correspond to
our $\,dt,w,\mathrm{Ric},(2-n)\dot f\,$ and $\,(2-n)\fh\nh$. On the
pseu\-\hbox{do\hs-}Riem\-ann\-i\-an
universal covering space $\,(\hm\nh,\hg)\,$ of $\,(M\nh,\g)\,$ the pull\-back
of $\,w\,$ with (a), still denoted by $\,w$, equals $\,\nabla\hn t$, and
$\,M=\hm\nnh/\hh\Gm\nh$, as in (\ref{wnt}) and (\ref{ucp}). Milnor's argument
\cite[p.\ 12]{milnor} is used in
\cite[proof of Lemma 1.2]{derdzinski-roter-08} to show that the levels of
$\,t\,$ in $\,\hm\,$ are connected: the word `locally' 
in (\ref{flt}) may be skipped. Due to $\,\Gm\nh$-in\-var\-i\-ance of
$\,w=\nabla\hn t$, one has (\ref{tqp}) with $\,q=1$, and the first line of
(\ref{fcg}) gives periodicity of $\,f\,$ as a function of $\,t\,$ (since
a nontrivial element of $\,\Gm\hs$ must thus have $\,p\ne0$). Nonconstancy of
$\,f\nh$,
cf.\ (\ref{ric}), implies that the values of $\,p\,$ arising from $\,\Gm\hs$
form a cyclic sub\-group of $\,\bbR\,$ 
with a unique 
generator $\,c>0$. Thus,
\begin{enumerate}
\item[(b)] $t:\hm\to\bbR\,$ descends to a bundle projection 
$\,M\to\bbR/c\hh\bbZ\hh=S^1\nnh$,
\end{enumerate}
which completes the proof of
\cite[Theorem B]{derdzinski-roter-08}.

As we already pointed out, the above proof remains valid if the
Lo\-rentz\-i\-an hypothesis is replaced by the weaker condition (a),
transversal orientability of $\,\mathcal{D}^\perp\nnh$ being always assumed.
However, a compact rank-one ECS manifold does not have to be {\it
translational\/} in the sense of satisfying (a) -- in 
other words, $\,q\ne1\,$ may occur in (\ref{fcg}), and then
$\,M=\hm\nnh/\hh\Gm\nh$, with its ECS metric, is referred to as {\it
di\-la\-tion\-al}. The existence of di\-la\-tion\-al-type compact rank-one ECS 
manifolds, including lo\-cal\-ly-ho\-mo\-ge\-ne\-ous ones, was established
quite recently in \cite[Theorems 6.1 and B.1]{derdzinski-terek-cl} --
and for them, instead of (b), one gets the different conclusion (c) appearing
below.

Compared to the above derivation of (b), the path leading to our proof of
Theorem~\ref{maith} is rather indirect, and we outline it here by briefly 
summarizing, in the following five sentences, the five paragraphs (or
two\hs-par\-a\-graph parts) of Section~\ref{om} that begin with the 
phrases `This is achieved', `Returning', `Finally', `First' and `On the other
hand'. Namely, we start by showing that $\,\mathcal{D}^\perp\nnh$ satisfies 
condition (\ref{cpl}). To derive (\ref{cpl}) for $\,\mathcal{D}^\perp\nnh$ 
we use the fact that -- since in $\,\hm\,$ the leaves of 
$\,\mathcal{D}^\perp\nnh$ are the connected components of levels of $\,t\,$ 
and $\,dt\,$ is parallel -- the leaf holonomy of any compact leaf of 
$\,\mathcal{D}^\perp\nnh$ may be dif\-feo\-mor\-phic\-al\-ly identified with 
its nor\-mal-con\-nec\-tion holonomy. Next, we introduce a vector space 
$\,\mathcal{F}$ of functions $\,\hm\to\bbR$, obviously having {\it either a
finite or an infinite dimension}. The former italicized case implies local 
homogeneity. The latter one causes a natural linear operator 
$\,\mathcal{F}\to H^1\nh(M\nh,\hs\bbR)$ to be noninjective, and
a nontrivial function lying in its kernel leads, via Sard's theorem, to
a compact leaf of $\,\mathcal{D}^\perp\nnh$ satisfying the second 
option in the ei\-ther-or clause of (\ref{cpl}), thus implying compactness
of all leaves of $\,\mathcal{D}^\perp\nnh$ and the conclusion that they form 
the fibres of a bundle projection $\,M\nh\to S^1\nh$.

A final remark: our proof of Theorem~\ref{maith} does
eventually lead to a conclusion analogous to (b), but different from it -- 
namely, in the
di\-la\-tion\-al case, unless $\,(M\nh,\g)$ is locally homogeneous, with
$\,\mathcal{D}^\perp\nnh$ still assumed transversally orientable,
\begin{enumerate}
\item[(c)] $\taw:\hm\to\bbR\,$ descends to a bundle projection 
$\,M\to\bbR/\bbZ\hh=S^1\nnh$.
\end{enumerate}
Here our choice of $\,t\,$ has been modified by an af\-fine substitution so
that the range $\,t(\hm)\,$ equals $\,(0,\infty)$, and
$\,\taw=(\log t)/(\log q)\,$ with suitably chosen
$\,q\in(0,\infty)\smallsetminus\{1\}$. Cf.\
\cite[the first paragraph proof of the of Theorem 2.3]{derdzinski-terek-ro}.

\end{document}